\RequirePackage{fix-cm}
\documentclass[smallextended]{svjour3}
 
\smartqed  

\usepackage{graphicx}
\usepackage{lipsum}
\usepackage{amsfonts}
\usepackage{graphicx}
\usepackage{epstopdf}
\usepackage{algorithm}
\usepackage{algorithmic}
\usepackage{tikz}
\usetikzlibrary{arrows}
\usepackage{hyperref}
\usepackage{cite}
\usepackage{multirow}
\usepackage{listings}
\usepackage{mathtools}
\usepackage{latexsym,amsmath,amsfonts,amscd}
\usepackage{subfigure}
\usepackage{bbm}

\newtheorem{Proposition}{Proposition}
\newtheorem{rmk}{Remark}
\newtheorem{defi}{Definition}

\ifpdf
  \DeclareGraphicsExtensions{.eps,.pdf,.png,.jpg}
\else
  \DeclareGraphicsExtensions{.eps}
\fi

\newcommand{\ac}{$\frac{1}{100}$}
\newcommand{\bc}{$\frac{1}{100}$}
\newcommand{\cc}{$\frac{1}{2}$}
\newcommand{\dc}{$\frac{1}{2}$}

\newcommand{\Addkz}{$\mathbf{26}$}
\newcommand{\Addko}{$\frac{-15\sqrt{5}-25}{8}$}
\newcommand{\Addkt}{$\frac{15\sqrt{5}-25}{8}$}
\newcommand{\Addkth}{$-\frac{1}{4}$}

\newcommand{\Addez}{$\mathbf{18}$}
\newcommand{\Addeo}{$\frac{-15\sqrt{5}-25}{8}$}
\newcommand{\Addet}{$\frac{15\sqrt{5}-25}{8}$}
\newcommand{\Addeth}{$-\frac{1}{4}$}
\newcommand{\Addef}{$\frac{3\sqrt{5}-5}{4}$}
\newcommand{\Addes}{$-\frac{5}{2}$}
\newcommand{\Addese}{$\frac{-3\sqrt{5}-5}{4}$}

\newcommand{\Addiz}{$\mathbf{10}$}
\newcommand{\Addif}{$\frac{3\sqrt{5}-5}{4}$}
\newcommand{\Addis}{$-\frac{5}{2}$}
\newcommand{\Addise}{$\frac{-3\sqrt{5}-5}{4}$}

\newcommand{\Aoddkz}{$\mathbf{26}$}
\newcommand{\Aoddko}{$\frac{-15\sqrt{5}-25}{8}$}
\newcommand{\Aoddkt}{$\frac{15\sqrt{5}-25}{8}$}
\newcommand{\Aoddkth}{$0$}

\newcommand{\Aoddez}{$\mathbf{17}$}
\newcommand{\Aoddeo}{$-7$}
\newcommand{\Aoddet}{$0$}
\newcommand{\Aoddeth}{$0$}
\newcommand{\Aoddef}{$0$}
\newcommand{\Aoddes}{$-\frac{5}{2}$}
\newcommand{\Aoddese}{$-$\bc}

\newcommand{\Aoddiz}{$\mathbf{10}$}
\newcommand{\Aoddif}{$0$}
\newcommand{\Aoddis}{$-$\cc}
\newcommand{\Aoddise}{$-$\dc}

\newcommand{\Moddkz}{$\mathbf{30}$}
\newcommand{\Moddko}{$\frac{-15\sqrt{5}-25}{8}$}
\newcommand{\Moddkt}{$0$}
\newcommand{\Moddkth}{$0$}

\newcommand{\Moddez}{$\mathbf{17}$}
\newcommand{\Moddeo}{$-7$}
\newcommand{\Moddet}{$0$}
\newcommand{\Moddeth}{$0$}
\newcommand{\Moddef}{$0$}
\newcommand{\Moddes}{$-$\ac}
\newcommand{\Moddese}{$-$\bc}

\newcommand{\Moddiz}{$\mathbf{10}$}
\newcommand{\Moddif}{$0$}
\newcommand{\Moddis}{$-$\cc}
\newcommand{\Moddise}{$-$\dc}

\newcommand{\Asoddkz}{$\mathbf{0}$}
\newcommand{\Asoddko}{$0$}
\newcommand{\Asoddkt}{$0$}
\newcommand{\Asoddkth}{$0$}

\newcommand{\Asoddez}{$\mathbf{0}$}
\newcommand{\Asoddeo}{$0$}
\newcommand{\Asoddet}{$0$}
\newcommand{\Asoddeth}{$0$}
\newcommand{\Asoddef}{$0$}
\newcommand{\Asoddes}{$-\frac{5}{2} +$ \ac}
\newcommand{\Asoddese}{$0$}

\newcommand{\Asoddiz}{$\mathbf{0}$}
\newcommand{\Asoddif}{$0$}
\newcommand{\Asoddis}{$0$}
\newcommand{\Asoddise}{$0$}

\newcommand{\MoLoddkz}{$\mathbf{30}$}
\newcommand{\MoLoddko}{$\frac{-15\sqrt{5}-25}{8}$}
\newcommand{\MoLoddkt}{$\frac{747\sqrt{5}+1245}{2720}$}
\newcommand{\MoLoddkth}{$0$}

\newcommand{\MoLoddez}{$\mathbf{17+\frac{249}{170000}}$}
\newcommand{\MoLoddeo}{$-7$}
\newcommand{\MoLoddet}{$0$}
\newcommand{\MoLoddeth}{$0$}
\newcommand{\MoLoddef}{$0$}
\newcommand{\MoLoddes}{$-\frac{5}{2}$}
\newcommand{\MoLoddese}{$-$\bc}

\newcommand{\MoLoddiz}{$\mathbf{10}$}
\newcommand{\MoLoddif}{$0$}
\newcommand{\MoLoddis}{$-$\cc}
\newcommand{\MoLoddise}{$-$\dc}
\newcommand{\MoLoddie}{$\frac{249}{3400}$}

\newcommand{\Atddkz}{$\mathbf{26}$}
\newcommand{\Atddko}{$\frac{-15\sqrt{5}-25}{8}$}
\newcommand{\Atddkt}{$\frac{15\sqrt{5}-25}{8}$}
\newcommand{\Atddkth}{$-\frac{1}{4}$}

\newcommand{\Atddez}{$\mathbf{17}$}
\newcommand{\Atddeo}{$-7$}
\newcommand{\Atddet}{$\frac{15\sqrt{5}-25}{8}$}
\newcommand{\Atddeth}{$0$}
\newcommand{\Atddef}{$\frac{3\sqrt{5}-5}{4}$}
\newcommand{\Atddes}{$-\frac{5}{2}$}
\newcommand{\Atddese}{$\frac{-3\sqrt{5}-5}{4}$}
\newcommand{\Atdden}{$\frac{1}{4}$}

\newcommand{\Atddiz}{$\mathbf{10}$}
\newcommand{\Atddif}{$0$}
\newcommand{\Atddis}{$-\frac{5}{2}$}
\newcommand{\Atddise}{$-$\cc}
\newcommand{\Atddie}{$-\frac{5}{14}$}

\newcommand{\Astddkz}{$\mathbf{0}$}
\newcommand{\Astddko}{$0$}
\newcommand{\Astddkt}{$0$}
\newcommand{\Astddkth}{$0$}

\newcommand{\Astddez}{$\mathbf{0}$}
\newcommand{\Astddeo}{$0$}
\newcommand{\Astddet}{$0$}
\newcommand{\Astddeth}{$0$}
\newcommand{\Astddef}{$0$}
\newcommand{\Astddes}{$0$}
\newcommand{\Astddese}{$\frac{-3\sqrt{5}-5}{4} +$\bc}

\newcommand{\Astddiz}{$\mathbf{0}$}
\newcommand{\Astddif}{$0$}
\newcommand{\Astddis}{$-\frac{5}{2} +$\cc}
\newcommand{\Astddise}{$0$}
\newcommand{\Astddie}{$-\frac{5}{14}$}

\newcommand{\AoLtddkz}{$\mathbf{26+4(\frac{747\sqrt{5}+1745}{2720})}$}
\newcommand{\AoLtddko}{$\frac{-15\sqrt{5}-25}{8}$}
\newcommand{\AoLtddkt}{$\frac{15\sqrt{5}-25}{8}$}
\newcommand{\AoLtddkth}{$\frac{3\sqrt{5}-505}{2720}$}

\newcommand{\AoLtddez}{$\mathbf{17}$}
\newcommand{\AoLtddeo}{$-7$}
\newcommand{\AoLtddet}{$\frac{7}{5}$}
\newcommand{\AoLtddeth}{$0$}
\newcommand{\AoLtddef}{$\frac{75\sqrt{5}+124}{680}$}
\newcommand{\AoLtddes}{$-\frac{5}{2}+2(\frac{1}{4})$}
\newcommand{\AoLtddese}{$\frac{-3\sqrt{5}-5}{4}$}
\newcommand{\AoLtdden}{$\frac{1}{4}$}

\newcommand{\AoLtddiz}{$\mathbf{10+2(\frac{1}{10})}$}
\newcommand{\AoLtddif}{$0$}
\newcommand{\AoLtddis}{$-\frac{5}{2}$}
\newcommand{\AoLtddise}{$-$\dc$+\frac{1}{56}$}

\newcommand{\AoLtddie}{$-\frac{5}{14}$}
\newcommand{\AoLtddite}{$2(\frac{1}{10})$}
\newcommand{\AoLtddiel}{$\frac{1}{56}$}
\newcommand{\AoLtdditw}{$2(\frac{75\sqrt{5}+124}{3400})$}

\newcommand{\Asthddkz}{$\mathbf{0}$}
\newcommand{\Asthddko}{$0$}
\newcommand{\Asthddkt}{$0$}
\newcommand{\Asthddkth}{$0$}

\newcommand{\Asthddez}{$\mathbf{0}$}
\newcommand{\Asthddeo}{$0$}
\newcommand{\Asthddet}{$0$}
\newcommand{\Asthddeth}{$0$}
\newcommand{\Asthddef}{$0$}
\newcommand{\Asthddes}{$0$}
\newcommand{\Asthddese}{$0$}

\newcommand{\Asthddiz}{$\mathbf{0}$}
\newcommand{\Asthddif}{$0$}
\newcommand{\Asthddis}{$0$}
\newcommand{\Asthddise}{$-2$}

\newcommand{\AtLthddkz}{$\mathbf{26}$}
\newcommand{\AtLthddko}{$\frac{-15\sqrt{5}-25}{8}$}
\newcommand{\AtLthddkt}{$\frac{15\sqrt{5}-25}{8}$}
\newcommand{\AtLthddkth}{$-\frac{1}{4}$}

\newcommand{\AtLthddez}{$\mathbf{17+2(\frac{7}{5})}$}
\newcommand{\AtLthddeo}{$-7$}
\newcommand{\AtLthddet}{$\frac{15\sqrt{5}-25}{8}$}
\newcommand{\AtLthddeth}{$\frac{-3\sqrt{5}+5}{8}$}
\newcommand{\AtLthddef}{$\frac{3\sqrt{5}-5}{4}$}
\newcommand{\AtLthddes}{$-\frac{5}{2}$}
\newcommand{\AtLthddese}{$\frac{-3\sqrt{5}-5}{4}$}
\newcommand{\AtLthdden}{$\frac{1}{4}+\frac{-3\sqrt{5}+5}{8}$}
\newcommand{\AtLthddethirt}{$\frac{7}{5}$}

\newcommand{\AtLthddiz}{$\mathbf{10}$}
\newcommand{\AtLthddif}{$\frac{1}{2}$}
\newcommand{\AtLthddis}{$-\frac{5}{2}$}
\newcommand{\AtLthddise}{$-2-$\dc}

\newcommand{\AtLthddie}{$-\frac{5}{14}+\frac{1}{2}$}
\newcommand{\AtLthddite}{$0$}
\newcommand{\AtLthddiel}{$0$}
\newcommand{\AtLthdditw}{$0$}

\newcommand{\Adkz}{\mathbf{13}}
\newcommand{\Adko}{\frac{-15\sqrt{5}-25}{8}}
\newcommand{\Adkt}{\frac{15\sqrt{5}-25}{8}}
\newcommand{\Adkth}{-\frac{1}{4}}

\newcommand{\Adiz}{\mathbf{5}}
\newcommand{\Adif}{\frac{3\sqrt{5}-5}{4}}
\newcommand{\Adis}{-\frac{5}{2}}
\newcommand{\Adise}{\frac{-3\sqrt{5}-5}{4}}

\newcommand{\Aodkz}{\mathbf{13}}
\newcommand{\Aodko}{-7}
\newcommand{\Aodkt}{\frac{15\sqrt{5}-25}{8}}
\newcommand{\Aodkth}{0}

\newcommand{\Aodiz}{\mathbf{4.8}}
\newcommand{\Aodif}{0}
\newcommand{\Aodis}{-2}
\newcommand{\Aodise}{-\frac{1}{2}}

\newcommand{\Modkz}{\mathbf{15}}
\newcommand{\Modko}{-7}
\newcommand{\Modkt}{0}
\newcommand{\Modkth}{0}

\newcommand{\Modiz}{\mathbf{4.8}}
\newcommand{\Modif}{0}
\newcommand{\Modis}{-\frac{1}{2}}
\newcommand{\Modise}{-\frac{1}{2}}

\newcommand{\Asodkz}{\mathbf{0}}
\newcommand{\Asodko}{0}
\newcommand{\Asodkt}{0}
\newcommand{\Asodkth}{0}

\newcommand{\Asodiz}{\mathbf{0}}
\newcommand{\Asodif}{0}
\newcommand{\Asodis}{-2+\frac{1}{2}}
\newcommand{\Asodise}{0}

\newcommand{\MoLodkz}{\mathbf{15}}
\newcommand{\MoLodko}{-7}
\newcommand{\MoLodkt}{\frac{35}{16}}
\newcommand{\MoLodkth}{0}

\newcommand{\MoLodiz}{\mathbf{4.8+\frac{5}{32}}}
\newcommand{\MoLodif}{0}
\newcommand{\MoLodis}{-2}
\newcommand{\MoLodise}{-\frac{1}{2}}

\newcommand{\Astdkz}{\mathbf{0}}
\newcommand{\Astdko}{0}
\newcommand{\Astdkt}{0}
\newcommand{\Astdkth}{0}

\newcommand{\Astdiz}{\mathbf{0}}
\newcommand{\Astdif}{0}
\newcommand{\Astdis}{-\frac{1}{2}}
\newcommand{\Astdise}{-\frac{11}{10}}

\newcommand{\AoLtdkz}{\mathbf{13+2(\frac{77}{48})}}
\newcommand{\AoLtdko}{-7+\frac{-75\sqrt{5}+125}{384}}
\newcommand{\AoLtdkt}{\frac{15\sqrt{5}-25}{8}+\frac{35}{48}}
\newcommand{\AoLtdkth}{\frac{-165\sqrt{5}+275}{384}}

\newcommand{\AoLtdiz}{\mathbf{4.8+\frac{5}{24}}}
\newcommand{\AoLtdif}{\frac{11}{24}}
\newcommand{\AoLtdis}{-\frac{5}{2}}
\newcommand{\AoLtdise}{-\frac{8}{5}}

\newcommand{\tabletextwidtho}{1}

\begin{document}

\title{On the  monotonicity of $Q^3$ spectral element method for  Laplacian}


\titlerunning{Monotonicity of $Q^3$ spectral element method for Laplacian}        

\author{Logan J. Cross       \and
        Xiangxiong Zhang  
}

\authorrunning{L. Cross and X. Zhang} 

\institute{L. J. Cross  \at \email{logancross68@gmail.com}          
           \and
          Corresponding author: X. Zhang (ORCID 0000-0002-1090-7189) \at
              \email{zhan1966@purdue.edu}\\
              Department of Mathematics,\\
Purdue University,\\
150 N. University Street,\\
West Lafayette, IN 47907-2067, USA.\\
            }
 

\maketitle

\begin{abstract}
The monotonicity of discrete Laplacian, i.e., inverse positivity of stiffness matrix, implies discrete maximum principle, which is in general not true for high order accurate schemes on unstructured meshes. 
On the other hand, it is possible to construct high order accurate monotone schemes on structured meshes. 
All previously known high order accurate inverse positive schemes  are or can be regarded as fourth order accurate finite difference schemes, which is 
 either an M-matrix or a product of two M-matrices. 
For the $Q^3$ spectral element method for the two-dimensional Laplacian, we prove its stiffness matrix is a product of four M-matrices thus it is unconditionally monotone. Such a scheme can be regarded as a fifth order accurate finite difference scheme.  
Numerical tests suggest that the unconditional monotonicity of $Q^k$ spectral element methods will be lost for $k\geq 9$ in two dimensions, and for $k\geq 4$ in three dimensions. In other words, for obtaining a high order monotone scheme, only $Q^2$ and $Q^3$ spectral element methods can be unconditionally monotone in three dimensions. 
\end{abstract}



\vspace{-0.5cm}


\section{Introduction}
\subsection{Monotone high order schemes}
In many applications, monotone discrete Laplacian operators are desired and useful for ensuring stability such as discrete maximum principle \cite{ciarlet1970discrete} or positivity-preserving of physically positive quantities \cite{shen2021discrete, hu2021positivity, liu2022structure}. Let $\Delta_h$ denote the matrix representation of a discrete Laplacian operator, then it is called {\it monotone} if $(-\Delta_h)^{-1}\geq 0$, i.e., the matrix $(-\Delta_h)^{-1}$ has nonnegative entries. In this paper, all inequalities for matrices are entry-wise inequalities. It is well known that the simplest second order accurate centered finite difference scheme $u''(x_i)\approx \frac{u(x_{i-1})-2u(x_i)+u(x_{i+1})}{\Delta x^2}$ is monotone because the corresponding matrix  $-\Delta_h$ is an M-matrix thus inverse positive. The most general extension of this result is to state that a  linear finite element method with special implementation under a mild mesh constraint forms an M-matrix   thus monotone on unstructured triangular meshes \cite{xu1999monotone}. 

 The discrete maximum principle is not true for high order finite element methods  on unstructured meshes \cite{hohn1981some}. On structured meshes, there exist a few high order accurate inverse positive finite difference schemes. 
To the best of our knowledge, the following schemes  for solving a two-dimensional Poisson equation are the only ones proven to be monotone  beyond the second order accuracy, and all of them can be regarded as finite difference schemes with four order accuracy for function values for solving elliptic and parabolic equations:
\begin{enumerate}
 \item  The classical 9-point scheme \cite{ krylov1958approximate, collatz1960, bramble1962formulation} 
 are monotone because the stiffness matrix is an M-matrix.  
 \item
In \cite{bramble1964finite,bramble1964new},  a fourth order accurate finite difference scheme
was constructed. The stiffness matrix is a product of two M-matrices thus monotone. 

\item The Lagrangian $P^2$  finite element method  on a regular triangular mesh \cite{whiteman1975lagrangian} has a monotone stiffness matrix \cite{lorenz1977inversmonotonie}. On an equilateral triangular mesh, the discrete maximum principle can also be proven \cite{hohn1981some}. It can be regarded as a finite difference scheme  at vertices and edge centers,
on which superconvergence of fourth order accuracy holds. 
\item Monotonicity was  proven for the $Q^2$ spectral element method on an uniform rectangular mesh  for a variable coefficient Poisson equation under suitable mesh constraints \cite{li2019monotonicity}. This scheme can be regarded as a fourth order accurate finite difference scheme \cite{li2019fourth, MR4378595}.
See also \cite{zhang2024monotonicity} for extensions to quasi-uniform meshes. 
\end{enumerate}

For solving $-\Delta u=f$ with homogeneous Dirichlet boundary condition on a rectangular domain, all schemes above can be written in the form $S\mathbf u=M\mathbf f$ with stiffness matrix $S^{-1}\geq 0$ and mass matrix $M\geq 0$,  thus $(-\Delta_h)^{-1}=S^{-1}M\geq 0$. 
The last two methods are finite difference schemes constructed from the variational formulation, thus they do not suffer from the drawbacks of the first two conventional finite difference schemes, such as
loss of accuracy on quasi-uniform meshes,  difficulty with other types of boundary conditions such as Neumann boundary, etc. 

\subsection{Monotonicity of $Q^k$ spectral element method}

The Lagrangian $Q^k$ continuous finite element
method on rectangular meshes implemented by $(k+1)$-point Gauss-Lobatto quadrature is often referred to as the spectral element method in the literature \cite{maday1990optimal}, which has been a very popular high order accurate method for more than three decades for various second order equations such as the wave equations \cite{cohen2001higher}. 
In this paper we are interested in the monotonicity of the $Q^k$ spectral element method for solving the Poisson equation $-\Delta u=f$. For a one-dimensional problem, the stiffness matrix in $Q^k$ spectral element method reduces to the stiffness matrix of the $P^k$ finite element method without any quadrature, for which the discrete maximum principle for arbitrary $k$ can be proven by discrete Green's function \cite{vejchodsky2007discrete}. For two-dimensional problems, $Q^2$ spectral element method was proven monotone in \cite{li2019monotonicity}. 

The $P^k$ finite element method with $(k+1)$-point Gauss-Lobatto quadrature for a one-dimensional problem $-u''(x)=f$ can be equivalently written as a finite difference scheme at all  Gauss-Lobatto quadrature points \cite{li2019fourth}, and for homogeneous Dirichlet boundary its matrix-vector form can be written as 
$S\mathbf u=M \mathbf f$, where $\mathbf u$ and $\mathbf f$ are vectors of function point values, 
$M$ is the lumped mass matrix, $S$ is the stiffness matrix.   The stiffness matrix of $Q^k$ spectral element method for $-u_{xx}-u_{yy}=f$ on a rectangular mesh with homogeneous Dirichlet boundary can be written as $S\otimes M+M\otimes S$. The result in  \cite{vejchodsky2007discrete} implies $S^{-1}M\geq 0$ for arbitrary polynomial degree $k$ on a uniform mesh in one dimension, thus it might seem natural to conjecture that  monotonicity $S\otimes M+M\otimes S$ holds also for arbitrary polynomial degree $k$ on uniform rectangular meshes in two dimensions.
However,  the monotonicity $(S\otimes M+M\otimes S)^{-1}\geq 0$ is simply not true for $Q^k$ element for $k\geq 9$, as shown  by numerical tests in Section \ref{sec:5}. 

Thus an interesting question is whether $Q^k$ spectral element method is monotone for 2D Laplacian. The $Q^2$ case was proven 
in \cite{li2019monotonicity}. In this paper, we will prove the monotonicity of the $Q^3$ case.  In two dimensions, the cases for $Q^k$ with $4\leq k\leq 8$ remain open. 

For 3D Laplacian, the numerical tests in Section \ref{sec:5} suggest that $Q^k$ spectral element method with $k\geq 4$ cannot be uncoditionally monotone. In other words, the stiffness matrix in $Q^k$ spectral element method for $-u_{xx}-u_{yy}-u_{zz}$  is
$$S\otimes M\otimes M+M\otimes S\times M+M\otimes M\times S,$$
which is no longer monotone  
when $k\geq 4$ as suggested by numerical tests. Notice that the proof of the monotonicity of $Q^2$ spectral element method for 2D Laplacian 
in \cite{li2019monotonicity} can be easily extended to the three dimensional case. Thus from this perspective, it is also interesting to study $Q^3$ spectral element method. 
\subsection{Contribution and organization of the paper}

For proving inverse positivity, the main viable tool in the literature is to use M-matrices which are inverse positive.
A convenient sufficient condition of M-matrices is to require all off-diagonal entries to be non-positive. Except   the fourth order compact finite difference,
all high order accurate schemes induce positive off-diagonal entries, destroying such a structure, which is a major challenge of proving monotonicity. In  \cite{bramble1964new} and \cite{bohl1979inverse}, and also the appendix  in \cite{li2019monotonicity}, M-matrix factorizations of the form $(-\Delta_h)^{-1}=M_1M_2$ were shown for special high order schemes but these M-matrix factorizations seem ad hoc and do not apply to other schemes or other equations. 
  In \cite{lorenz1977inversmonotonie},
Lorenz proposed some matrix entry-wise inequality for ensuring a matrix to be a product of two M-matrices and applied it to $P^2$ finite element method on uniform regular triangular meshes.  
In  \cite{li2019monotonicity}, Lorenz's condition was applied to $Q^2$ spectral element method  on uniform meshes.  See extensions to quasi-uniform meshes in \cite{zhang2024monotonicity}.

For $Q^k$ spectral element method with $k\geq 3$, it does not seem possible to apply Lorenz's condition directly. Instead, we will demonstrate that Lorenz's condition can be applied to a few auxiliary matrices to establish the monotonicity in $Q^3$ spectral element method, which can be regarded a fifth order accurate finite difference scheme \cite{li2019fourth, MR4378595}. 
To the best of our knowledge, this is the first time that monotonicity can be proven for a fifth order accurate scheme in two dimensions.   
We are able to show the fifth order $Q^3$ spectral element on a uniform mesh in two dimensions can be factored into a product of four M-matrices, whereas existing M-matrix factorizations for high order schemes involved products of only two M-matrices.  

The rest of the paper is organized as follows. In Section \ref{sec:2}, we briefly review the conventional monotone high order finite difference schemes. In Section \ref{sec:3}, we review the monotone $P^2$ and $Q^2$ finite element methods in their equivalent finite difference forms, which are fourth order accurate in an {\it a priori} error estimate of function values at finite difference grid points for a smooth solution. 
In Section \ref{sec-lorenz}, we review the Lorenz's condition for proving monotonicity.
In Section \ref{sec-q3}, we prove the monotonicity of $Q^3$  spectral element scheme on a uniform mesh. 
Some numerical tests of these schemes  are given in Section \ref{sec:5}. Section \ref{sec:6} are concluding remarks.

 \section{Classical monotone high order finite difference schemes}
 \label{sec:2} 
 \subsection{9-point scheme}
 The 9-point scheme was somewhat suggested already in 1940s \cite{fox1947some} and discussed in details in \cite{collatz1960, krylov1958approximate}.
 It can be extended to higher dimensions
  \cite{ bramble1962formulation, bramble1963fourth}.

Consider solving the two-dimensional Poisson equations $-u_{xx}-u_{yy}=f$ with homogeneous Dirichlet boundary conditions on 
a rectangular domain $\Omega{} = (0,1) \times{} (0,1)$.  
Let $u_{i,j}$ denote the numerical solutions at a uniform grid $(x_i, y_j)=(\frac{i}{Nx}, \frac{j}{Ny})$, and $f_{i,j}=f(x_i, y_j)$. 
 For convenience, we introduce two matrices,
\[U=\begin{pmatrix*}[l]
  u_{i-1, j+1} &u_{i, j+1} &u_{i+1, j+1} \\
  u_{i-1, j} &u_{i, j} &u_{i+1, j} \\
  u_{i-1, j-1} &u_{i, j-1} &u_{i+1, j-1}
 \end{pmatrix*}, \quad F=\begin{pmatrix*}[l]
  f_{i-1, j+1} &f_{i, j+1} &f_{i+1, j+1} \\
  f_{i-1, j} &f_{i, j} &f_{i+1, j} \\
  f_{i-1, j-1} &f_{i, j-1} &f_{i+1, j-1}
 \end{pmatrix*}.\]
  Then the 9-point discrete Laplacian  for the Poisson equation  at a grid point $(x_i, y_j)$ can be written as
\begin{equation}
\frac{1}{12\Delta x^2}\begin{pmatrix}
                       -1 & 2 & -1\\
                       -10 & 20 & -10\\
                       -1 & 2 & -1 \end{pmatrix} : U
           +\frac{1}{12\Delta y^2}\begin{pmatrix}
                       -1 & -10 & -1\\
                       2 & 20 & 2\\
                       -1 & -10 & -1 \end{pmatrix} : U  
   =
\frac{1}{12}
 \begin{pmatrix}
                       0 & 1 & 0\\
                       1 & 8 & 1\\
                       0 & 1 & 0 \end{pmatrix}:
 F. 
\label{poisson-scheme3} 
\end{equation}
where $:$ denotes the sum of all entry-wise products in two matrices of the same size.
Under the assumption $\Delta x=\Delta y=h$,
it reduces to the following: \begin{equation}
\frac{1}{6h^2} \begin{pmatrix}
                       -1 & -4 & -1\\
                       -4 & 20 & -4\\
                       -1 & -4 & -1 \end{pmatrix}  
:  U =
 \frac{1}{12}
 \begin{pmatrix}
                       0 & 1 & 0\\
                       1 & 8 & 1\\
                       0 & 1 & 0 \end{pmatrix}:
 F. 
 \label{9point1}
\end{equation}

 The 9-point scheme can also be regarded as a compact finite difference scheme \cite{fornberg2015primer}. There can exist a few or many different compact finite difference approximations of the same order \cite{lele1992compact}. For instance, with the fourth order compact finite difference approximation to Laplacian used in \cite{li2018high, li2023-compact}, we   get the following scheme:
 \begin{equation}\resizebox{\textwidth}{!} 
{$ \frac{1}{12\Delta x^2}\begin{pmatrix}
                       -1 & 2 & -1\\
                       -10 & 20 & -10\\
                       -1 & 2 & -1 \end{pmatrix} : U
           +\frac{1}{12\Delta y^2}\begin{pmatrix}
                       -1 & -10 & -1\\
                       2 & 20 & 2\\
                       -1 & -10 & -1 \end{pmatrix} : U  
  =
 \frac{1}{144}
 \begin{pmatrix}
                       1 & 10 & 1\\
                       10 & 100 & 10\\
                       1 & 10 & 1 \end{pmatrix}
                       :
  F.$}
\label{poisson-scheme2} 
\end{equation}
 Both schemes \eqref{poisson-scheme3}  and \eqref{poisson-scheme2} are fourth order accurate and they have the same stencil and the same stiffness matrix in the left hand side. We have not observed any significant difference in numerical performances between these two schemes.

   \begin{rmk}
   \label{remark-laplace}
For solving 2D Laplace equation $-\Delta u=0$ with Dirichlet boundary conditions,  the 9-point scheme becomes sixth order accurate 
  \cite{fornberg2015primer}.
 \end{rmk}

 Nonsingular M-matrices are inverse-positive matrices. There are many equivalent definitions or characterizations of M-matrices, see 
\cite{plemmons1977m}. 
The following is a convenient sufficient but not necessary characterization of nonsingular M-matrices \cite{li2019monotonicity}:
\begin{theorem}
\label{rowsumcondition-thm}
For a real square matrix $A$  with positive diagonal entries and non-positive off-diagonal entries, $A$ is a nonsingular M-matrix if  all the row sums of $A$ are non-negative and at least one row sum is positive. 
\end{theorem}
 By condition $K_{35}$ in \cite{plemmons1977m}, a sufficient and necessary characterization  is,
 \begin{theorem}
\label{rowsumcondition-thm2}
 For a real square matrix $A$  with positive diagonal entries and non-positive off-diagonal entries, $A$ is a nonsingular M-matrix if  and only if that there exists a positive diagonal matrix 
 $D$ such that $AD$ has all
positive row sums.
 \end{theorem}
\begin{rmk}
Non-negative row sum is not a necessary condition for M-matrices. For instance, the following matrix $A$ is an M-matrix by Theorem \ref{rowsumcondition-thm2}:
  $$A = \begin{pmatrix}
                       10  & 0 & 0\\
                       -10 & 2 & -10\\
                       0   & 0  & 10 \end{pmatrix} 
,D =
 \begin{pmatrix}
                     0.1 &   0 & 0\\
                       0 & 2   & 0\\
                       0 &   0 & 0.1 \end{pmatrix}
,AD =
 \begin{pmatrix}
                     1   &   0 & 0\\
                      -1 &  4  & -1\\
                       0 &   0 & 1 \end{pmatrix}.
$$
\end{rmk}

The stiffness matrix in the scheme \eqref{9point1} has diagonal entries $\frac{20}{6h^2}$ and offdiagonal entries $-\frac{1}{6h^2}$, $-\frac{4}{6h^2}$ and $0$, 
thus by Theorem \ref{rowsumcondition-thm} it is an M-matrix and the scheme is monotone. In order for the stiffness matrix in \eqref{poisson-scheme3} and \eqref{poisson-scheme2} to be an M-matrix, we need all the off-diagonal entries to be nonnegative, which is true under the mesh constraints 
$ \frac{1}{ \sqrt 5}\leq \frac{\Delta x}{\Delta y}\leq \sqrt 5.$

 \subsection{The Bramble and Hubbard's scheme}
 
 In \cite{bramble1964new}, a fourth order accurate monotone scheme was constructed. Consider solving a one-dimensional problem 
 \begin{equation}
 -u''=f ,\quad x\in[0,1],\quad u(0)=\sigma_0, u(1)=\sigma_1,
 \label{problem-1d}
 \end{equation} on a uniform grid $x_i=\frac{i}{n+1}$ ($i=0,1,\cdots, n+1$). The scheme can be written as
 \[ \frac{-\sigma_0+2u_1 -u_2}{\Delta x^2}=f_1, \quad \frac{-u_{n-1}+2u_n -\sigma_1}{\Delta x^2}=f_n\]  
 \[\frac{\frac{1}{12}u_{i-2}-\frac43 u_{i-1}+\frac52 u_i -\frac43 u_{i+1}+\frac{1}{12}u_{i+2}}{\Delta x^2}=f_i, \quad i=2,3,\cdots, n-1. \]
 The matrix vector form of the scheme is $\frac{1}{\Delta x^2} H\mathbf u=\tilde { \mathbf f}$ where
 \[ \resizebox{\textwidth}{!} 
{$ H=\begin{pmatrix}
              2 & -1 & & & &&\\
              -\frac43 & \frac52 &-\frac43 &\frac{1}{12} & &\\      
             \frac{1}{12}& -\frac43 & \frac52 &-\frac43 &\frac{1}{12} & \\        
              & \ddots &  \ddots &  \ddots &  \ddots&  \ddots& &\\
            & & \frac{1}{12}& -\frac43 & \frac52 &-\frac43 &\frac{1}{12}  \\       
             & & & \frac{1}{12} & -\frac43 & \frac52 &-\frac43    \\
            &   & & & & -1 & 2 \\
             \end{pmatrix}, \mathbf u=\begin{pmatrix} u_1 \\ u_2 \\ \\ \vdots\\ \\ u_{n-1}\\ u_n\end{pmatrix}, \tilde { \mathbf f}=\begin{pmatrix} f_1\\ f_2   & \\ \\ \vdots\\ \\ f_{n-1}  \\ f_n \end{pmatrix}+ \begin{pmatrix} \frac{\sigma_0}{\Delta x^2} \\  -\frac{\sigma_0}{12\Delta x^2} & \\ \\ \mathbf 0 \\ \\  -\frac{\sigma_1}{12\Delta x^2} \\ \frac{\sigma_1}{\Delta x^2} \end{pmatrix}.$}\]
 For two-dimensional Laplacian, the scheme is defined similarly. In particular, assume $\Delta x=\Delta y=h$ for a square domain, the stiffness matrix can be written as $\frac{1}{h^2}(H\otimes   I+ I\otimes H)$ where $I$ is the identity matrix and $\otimes$ is the Kronecker product. 
 Its monotonicity was proven in \cite{bramble1964new}. 
 
 \section{Monotone high order finite element methods on structured meshes}
 \label{sec:3}
 It is well-known that finite element methods with suitable quadrature are equivalent to finite difference schemes. The schemes in this section are equivalent to finite difference schemes defined at quadrature points. 
 \subsection{Finite element method with the simplest quadrature}
 Consider an elliptic equation on $\Omega=(0,1)\times(0,1)$ with 
Dirichlet boundary conditions:
\begin{equation}
\label{pde-3}
 \mathcal L u\equiv -\nabla\cdot(a \nabla u)+cu =f  \quad \mbox{on} \quad \Omega, \quad
 u=g \quad \mbox{on}\quad  \partial\Omega. 
\end{equation}
Assume there is a function $\bar g \in H^1(\Omega)$ as an extension of $g$ so that $\bar g|_{\partial \Omega} = g$.
 The  variational form  of \eqref{pde-3} is to find $\tilde u = u - \bar g \in H_0^1(\Omega)$ satisfying
\begin{equation}\label{nonhom-var}
 \mathcal A(\tilde u, v)=(f,v) - \mathcal A(\bar g,v) ,\quad \forall v\in H_0^1(\Omega),
 \end{equation}
 where $\mathcal A(u,v)=\iint_{\Omega} a \nabla u \cdot \nabla v dx dy+\iint_{\Omega} c u v dx dy$, $ (f,v)=\iint_{\Omega}fv dxdy.$

    \begin{figure}[htbp]
 \subfigure[The quadrature points and a finite element mesh for $P^2$]{\includegraphics[scale=0.8]{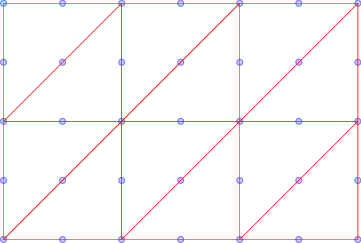} }
 \hspace{.6in}
 \subfigure[The corresponding finite difference grid]{\includegraphics[scale=0.8]{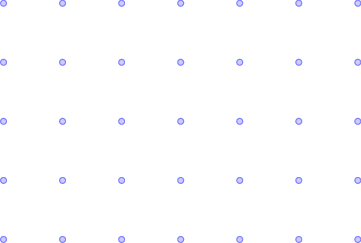}}
\caption{An illustration of Lagrangian $P^2$ element and the simple third order accurate quadrature using vertices and edge centers. }
\label{mesh3}
 \end{figure}
 
   \begin{figure}[h]
 \subfigure[The  quadrature points and a finite element mesh]{\includegraphics[scale=0.8]{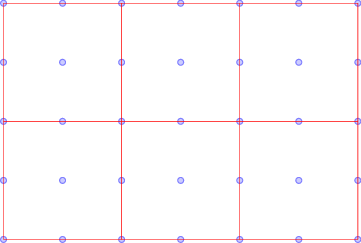} }
 \hspace{.6in}
 \subfigure[The corresponding finite difference grid]{\includegraphics[scale=0.8]{mesh1}}
\caption{An illustration of Lagrangian $Q^2$ element and the $3\times3$ Gauss-Lobatto quadrature. }
\label{mesh-Q2}
 \end{figure}

 Let $h$ be quadrature point spacing of  a regular triangular mesh shown in Figure \ref{mesh3} (or a  rectangular mesh shown in Figure \ref{mesh-Q2}) and  $V_0^h\subseteq H^1_0(\Omega)$ be the continuous finite element space consisting of piecewise $P^2$ polynomials (or $Q^2$ polynomials), then the most convenient implementation of  finite element method is to use the simple quadrature consisting of vertices and edge centers with equal weights (or $3\times3$ Gauss-Lobatto quadrature rule) for all the integrals,  see  Figure \ref{mesh3} for $P^2$ method (or Figure \ref{mesh-Q2} for $Q^2$ method). 
 Such a numerical scheme can be defined as:  find $ u_h \in V_0^h$ satisfying
\begin{equation}\label{nonhom-var-num3}
\mathcal A_h( u_h, v_h)=\langle f,v_h \rangle_h - \mathcal A_h( g_I,v_h) ,\quad \forall v_h\in V_0^h,
\end{equation}
where   $\mathcal A_h(u_h,v_h)$  and $\langle f,v_h\rangle_h$ denote using simple quadrature for integrals $\mathcal A(u_h,v_h)$ and $(f,v_h)$ respectively, and $g_I$ is the piecewise $P^2$ (or $Q^2$) Lagrangian interpolation polynomial at the quadrature points shown in 
 Figure \ref{mesh3} for $P^2$ method (or Figure \ref{mesh-Q2} for $Q^2$ method) of the following function:
\[g(x,y)=\begin{cases}
   0,& \mbox{if}\quad (x,y)\in (0,1)\times(0,1),\\
   g(x,y),& \mbox{if}\quad (x,y)\in \partial\Omega.\\
  \end{cases}
\] 

Then $\bar u_h =   u_h + g_I$ is the numerical solution for the problem \eqref{pde-3}.
Notice that \eqref{nonhom-var-num3} is not a straightforward approximation to \eqref{nonhom-var} since $\bar g$ is never used. 
When the numerical solution is represented by a linear combination of Lagrangian interpolation polynomials at the grid points, it can be rewritten as a finite difference scheme.
We also call it a variational difference scheme since it is derived from the
variational form. 
 
 \subsection{The $P^2$ finite element method}
For Laplacian $\mathcal L u=-\Delta u$, the scheme \eqref{nonhom-var-num3} on a uniform regular triangular mesh can be given in
a finite difference form \cite{whiteman1975lagrangian}:
 
\begin{subequations}
 \label{p2-scheme}
\begin{equation}
\label{fd_edge}
\begin{gathered}
 \frac{1}{h^2}\begin{pmatrix}
                       0 & -1 & 0\\
                       -1 & 4 & -1\\
                       0 & -1 & 0 \end{pmatrix} : U
  = f_{i,j},
  \quad
   \text{if $(x_i,y_j)$ is an edge center;}
   \end{gathered}
\end{equation}

\begin{equation} \label{fd_corner}
\begin{gathered}
 \frac{1}{9h^2}\begin{pmatrix}
                       1 & -4 & 1\\
                       -4 & 12 & -4\\
                       1 & -4 & 1 \end{pmatrix} : U
  = 0, \quad
    \text{if $(x_i,y_j)$ is a vertex.} 
\end{gathered}
\end{equation}
\end{subequations}

Notice that the stiffness matrix is not an M-matrix due to the positive off-diagonal entries in \eqref{fd_corner} and its inverse positivity was proven in \cite{lorenz1977inversmonotonie}. 

Since the  simple quadrature  is exact for integrating only quadratic polynomials on triangles, it is not obvious why
the finite  difference scheme  \eqref{p2-scheme} is fourth order accurate. With such a quadrature on two adjacent triangles forming a rectangle in a regular triangular mesh, we obtain a quadrature on the rectangle, see Figure \ref{p2_quad}.
  For a reference square $[-1,1]\times[-1,1]$, the quadrature weights are $\frac23$ and $\frac43$ for an edge center and the cell center respectively.

    \begin{figure}[h]
    \centering
\includegraphics[scale=1]{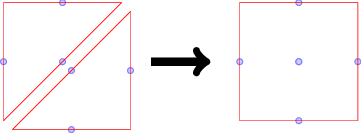}
 \caption{The simple quadrature on two triangles give a quadrature on a square.}
 \label{p2_quad}
 \end{figure}

 \begin{lemma}
 The quadrature  on a square $[-1,1]\times[-1,1]$ using only four edge centers  with weight $\frac23$  and one cell center with weight $\frac43$ is exact for $P^3$ polynomials.
 \end{lemma}
 \begin{proof}
  Since the quadrature is exact for integrating $P^2$ polynomials on either triangle in Figure \ref{p2_quad}, it suffices to show that it is exact for integrating basis polynomials of degree three, i.e., $x^2y$, $xy^2$, $x^3$ and $y^3$. It is straightforward to verify  that both exact integrals and quadrature of these four  polynomials on the square are  zero. 
 \end{proof}
 
Therefore, with Bramble-Hilbert Lemma (see Exercise 3.1.1 and Theorem 4.1.3 in \cite{ciarlet1978finite}), we can show that the quadrature rule is fourth order accurate  if we regard the regular triangular mesh in Figure \ref{p2_quad} (a) as a rectangular mesh.

The standard $L^2(\Omega)$-norm estimate for the finite element method with quadrature  \eqref{nonhom-var-num3}  using Lagrangian $P^2$ elements is third order accurate of function value for smooth exact solutions \cite{ciarlet1978finite}. 
On the other hand, superconvergence of function values in finite element method without quadrature can  be proven  \cite{chen2001structure, wahlbin2006superconvergence}, e.g., the errors at vertices and edge centers are fourth order accurate on triangular meshes for function values if using $P^2$ basis, see also \cite{huang2008superconvergence}.  It can be  shown that
using such fourth order accurate quadrature will not affect the fourth order superconvergence even for a general variable coefficient elliptic problem, see \cite{li2019fourth}. 
Notice that the scheme can also be given on a nonuniform mesh and its fourth order accuracy still holds on a quasi uniform mesh since it is also a finite element method.

\subsection{The $Q^2$ spectral element method}
The   scheme \eqref{nonhom-var-num3} with Lagrangian $Q^2$ basis is  fourth order accurate   \cite{li2019fourth} and  monotone on a uniform mesh under suitable mesh constraints \cite{li2019monotonicity}.  

 Consider a uniform grid $(x_i,y_j)$ for a rectangular domain $[0,1]\times[0,1]$
where
$x_i = ih$, $i = 0,1,\dots, n+1$
and
$y_j = jh$, $j = 0,1,\dots, n+1$, $h=\frac{1}{n+1}$, where $n$ must be odd.
Let $u_{ij}$ denote the numerical solution at $(x_i, y_j)$.
Let $\mathbf u$ denote an abstract vector consisting of $u_{ij}$ for $i,j=1,2,\cdots,n$. Let $\bar{\mathbf u}$ denote an abstract vector consisting of $u_{ij}$ for $i,j=0,1,2,\cdots,n,n+1$.
Let $\bar{\mathbf f}$ denote an abstract vector consisting of $f_{ij}$ for $i,j=1,2,\cdots,n$ and the boundary condition $g$ at the boundary grid points. 
Then the matrix vector representation of  \eqref{nonhom-var-num3} is
 $S\bar{\mathbf u}=M\mathbf f$ where $S$ is the stiffness matrix and $M$ is the lumped mass matrix.  
For convenience,  after inverting the mass matrix, with the boundary conditions, the whole scheme can be represented in a matrix vector form 
$\bar L_h \bar{\mathbf u} =\bar{\mathbf f}$.
For Laplacian $\mathcal L u=-\Delta u$,  $\bar L_h \bar{\mathbf u} =\bar{\mathbf f}$ on a uniform mesh is given as 
\begin{equation}
\label{Q2-laplacian}
\resizebox{\textwidth}{!} 
{$ 
    \begin{gathered}
(\bar L_h \bar{\mathbf u})_{i,j}:=\frac{-u_{i-1,j}-u_{i+1,j}+4u_{i,j}-u_{i,j+1}-u_{i,j-1}}{h^2}=f_{i,j},\quad \text{if $(x_i,y_j)$ is a cell center}, \\
(\bar L_h \bar{\mathbf u})_{i,j}:=\frac{-u_{i-1,j}+2u_{i,j}-u_{i+1,j}}{h^2}+\frac{u_{i,j-2}-8u_{i,j-1}+14u_{i,j}-8u_{i,j+1}+u_{i,j+2}}{4h^2}=f_{i,j},
\\ \text{if $(x_i,y_j)$ is an edge center for an edge parallel to the y-axis,}
\\ 
(\bar L_h \bar{\mathbf u})_{i,j}:=\frac{u_{i-2,j}-8u_{i-1,j}+14u_{i,j}-8u_{i+1,j}+u_{i+2,j}}{4h^2}+\frac{-u_{i,j-1}+2u_{i,j}-u_{i,j+1}}{h^2}=f_{i,j},
\\ \text{if $(x_i,y_j)$ is an edge center for an edge parallel to the x-axis,}
\\  
(\bar L_h \bar{\mathbf u})_{i,j}:=\frac{u_{i-2,j}-8u_{i-1,j}+14u_{i,j}-8u_{i+1,j}+u_{i+2,j}}{4h^2}+\frac{u_{i,j-2}-8u_{i,j-1}+14u_{i,j}-8u_{i,j+1}+u_{i,j+2}}{4h^2}=f_{i,j}, \\
 \text{if $(x_i,y_j)$ is a knot,}\\
(\bar L_h \bar{\mathbf u})_{i,j}:=u_{i,j}=g_{i,j}\quad \text{if $(x_i,y_j)$ is a boundary point.}
    \end{gathered}$}
\end{equation}
If ignoring the denominator $h^2$, then the stencil can be represented as:
 \[ \quad \mbox{  cell center}  \begin{array}{ccc}
& -1& \\
-1 & 4 & -1\\
& -1& 
\end{array}\qquad
\mbox{knots}  \begin{array}{ccccc}
&& \frac14& &\\
&& -2& &\\
\frac14& -2 & 7 & -2 &\frac14\\
&& -2& &\\
&& \frac14& &
\end{array}
\]
\[
\mbox{edge center (edge parallel to $y$-axis)}  \begin{array}{ccccc}
&& -1& &\\
\frac14& -2 & \frac{11}{2} & -2 &\frac14\\
&& -1& &
\end{array}\]
\[
\mbox{edge center (edge parallel to $x$-axis)}  \begin{array}{ccccc}
&& \frac14& &\\
&& -2& &\\
& -1 & \frac{11}{2} & -1 &\\
&& -2& &\\
&& \frac14& &
\end{array}\]

\subsection{The $Q^3$ spectral element method}

In the scheme \eqref{nonhom-var-num3}, if using Lagrangian $Q^3$ basis with $4\times 4$ Gauss-Lobatto quadrature, we get $Q^3$ spectral element method, which is also a fifth order accurate finite difference scheme
\cite{li2019fourth}.
The 4-point Gauss-Lobatto quadrature for  the reference interval $[-1, 1]$ has four quadrature points  $[-1 \: -\frac{\sqrt{5}}{5} \: \frac{\sqrt{5}}{5} \: 1]$. Thus on an uniform rectangular mesh, the corresponding finite difference grid consisting of quadrature points is not exactly uniform, see Figure \ref{mesh-Q3}.

 \begin{figure}[htbp]
      \subfigure[ Quadrature points and a finite element mesh.]{\includegraphics[scale=0.8]{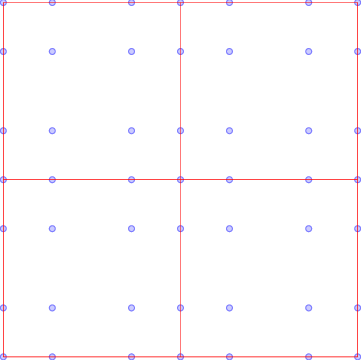} }
 \hspace{.6in}
 \subfigure[The corresponding finite difference grid. ]{\includegraphics[scale=0.8]{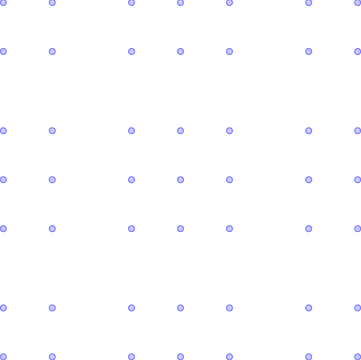}}
     \caption{An illustration of a mesh for $Q^3$ element and the $4\times4$ Gauss-Lobatto quadrature. } 
     \label{mesh-Q3}
\end{figure}

\begin{figure}[htbp]
    \label{q3_stencil_def_1D} \centering \includegraphics[scale=0.8]{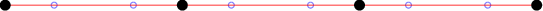}
    \caption{Three adjacent 1D cells for $P^3$ elements using 4-point Gauss-Lobatto quadrature.}
\end{figure}

Now consider a uniform mesh for a one-dimensional problem and assume each cell has length $h$,  see Figure \ref{q3_stencil_def_1D}. 
There are two quadrature points inside each interval, and we refer to them as the left interior point and the right interior point.  
The $Q^3$ spectral element method in difference form  for one-dimension problem \eqref{problem-1d} can be written as
$\bar L_h \bar{\mathbf u} =\bar{\mathbf f}$:
\begin{equation}
\label{q3-scheme-1d}
\resizebox{\textwidth}{!} 
{$
    \begin{gathered}
        ( \bar L_h \bar{\mathbf u})_i:=   \frac{4}{h^2}\left[13u_i - \frac{15\sqrt{5}+25}{8}(u_{i-1} + u_{i+1}) + \frac{15\sqrt{5}-25}{8}(u_{i-2} + u_{i+2}) - \frac{1}{4}(u_{i-3} + u_{i+3})\right]=f_i, \text{$x_i$ is a knot;}
\\ 
(\bar L_h \bar{\mathbf u})_i:=\frac{4}{h^2}\left[-\frac{3\sqrt{5}+5}{4}u_{i-1} + 5u_i + \frac{-5}{2}u_{i+1} +  \frac{15\sqrt{5}-25}{8}u_{i+2} \right]=f_i, \quad \text{ $x_i$ is the left interior point;} \\
  \\ 
  (\bar L_h \bar{\mathbf u})_i:=\frac{4}{h^2}\left[ \frac{15\sqrt{5}-25}{8}u_{i-2} - \frac{5}{2}u_{i-1} +  5u_i - \frac{3\sqrt{5}+5}{4}u_{i+1} \right]=f_i,\quad \text{if $x_i$ is the right interior point.}   \\
(\bar L_h \bar{\mathbf u})_0:=u_0=\sigma_0,\qquad (\bar L_h \bar{\mathbf u})_{n+1}:=u_{n+1}=\sigma_1.
    \end{gathered}
  $}
\end{equation}
The explicit scheme in two dimensions will be given in Section \ref{sec-q3}.

\section{Lorenz's condition for monotonicity} 
\label{sec-lorenz}
In this section, we briefly review the Lorenz's condition for monotonicity \cite{lorenz1977inversmonotonie}, which will be the main tool to prove the monotonicity of the $Q^3$ spectral element method. 
The monotonicity implies discrete maximum principle for the scheme, see \cite{ciarlet1970discrete, li2019monotonicity}.

\begin{defi}
Let $\mathcal N = \{1,2,\dots,n\}$. For $\mathcal N_1, \mathcal N_2 \subset \mathcal N$, we say a matrix $A=[a_{ij}]$ of size $n\times n$ connects $\mathcal N_1$ with $\mathcal N_2$ if 
\begin{equation}
\forall i_0 \in \mathcal N_1, \exists i_r\in \mathcal N_2, \exists i_1,\dots,i_{r-1}\in \mathcal N \quad \mbox{s.t.}\quad  a_{i_{k-1}i_k}\neq 0,\quad k=1,\cdots,r.
\label{condition-connect}
\end{equation}
If perceiving $A$ as a directed graph adjacency matrix of vertices labeled by $\mathcal N$, then \eqref{condition-connect} simply means that there exists a directed path from any vertex in $\mathcal N_1$ to at least one vertex in $\mathcal N_2$.  
In particular, if $\mathcal N_1=\emptyset$, then any matrix $A$  connects $\mathcal N_1$ with $\mathcal N_2$.
\end{defi}

Given a square matrix $A$ and a column vector $\mathbf x$, we define
\[\mathcal N^0(A\mathbf x)=\{i: (A\mathbf x)_i=0\},\quad 
\mathcal N^+(A\mathbf x)=\{i: (A\mathbf x)_i>0\}.\]

Given a matrix $A=[a_{ij}]\in \mathbbm{R}^{n\times n}$, define its diagonal, off-diagonal, positive and negative off-diagonal parts as $n\times n$ matrices $A_d$, $A_a$, $A_a^+$, $A_a^-$:
\[(A_d)_{ij}=\begin{cases}
a_{ii}, & \mbox{if} \quad i=j\\
0, & \mbox{if} \quad  i\neq j
\end{cases}, \quad A_a=A-A_d,
\]
\[(A_a^+)_{ij}=\begin{cases}
a_{ij}, & \mbox{if} \quad a_{ij}>0,\quad i\neq j\\
0, & \mbox{otherwise}.
\end{cases}, \quad A_a^-=A_a-A^+_a.
\]

The following two results were proven in  \cite{lorenz1977inversmonotonie}. See also \cite{li2019monotonicity} for a detailed proof.
\begin{theorem}\label{thm2}
If $A\leq M_1M_2\cdots M_k L$ where $M_1, \cdots, M_k$ are nonsingular M-matrices and $L_a\leq 0$,  and there exists a nonzero vector $\mathbf e\geq 0$ such that one of the matrices $M_1, \cdots, M_k , L$ connects $\mathcal N^0(A\mathbf e)$ with $\mathcal N^+(A\mathbf e)$. Then $M_k^{-1}M_{k-1}^{-1}\cdots M_1^{-1} A$ is an M-matrix, thus $A$ is a product of $k+1$ nonsingular M-matrices and $A^{-1}\geq 0$. 
\end{theorem}
\begin{theorem}[Lorenz's condition] \label{thm3}
If $A^-_a$ has a decomposition: $A^-_a = A^z + A^s = (a_{ij}^z) + (a_{ij}^s)$ with $A^s\leq 0$ and $A^z \leq 0$, such that 
\begin{subequations}
 \label{lorenz-condition}
\begin{align}
& A_d + A^z \textrm{ is a nonsingular M-matrix},\label{cond1}\\ 
& A^+_a \leq A^zA^{-1}_dA^s \textrm{ or equivalently } \forall a_{ij} > 0 \textrm{ with } i \neq j, a_{ij} \leq \sum_{k=1}^n a_{ik}^za_{kk}^{-1}a_{kj}^s,\label{cond2}\\
& \exists \mathbf e \in \mathbbm{R}^n\setminus\{\mathbf 0\}, \mathbf e\geq 0 \textrm{ with $A\mathbf e \geq 0$ s.t. $A^z$ or $A^s$  connects $\mathcal N^0(A\mathbf e)$ with $\mathcal N^+(A\mathbf e)$.} \label{cond3}
\end{align}
\end{subequations}
Then $A$ is a product of two nonsingular M-matrices thus $A^{-1}\geq 0$.
\end{theorem}

The following result can be found in \cite{zhang2024monotonicity}: 
\begin{Proposition} The matrix $L$ in Theorem \ref{thm2} must be an M-matrix.
\end{Proposition}

In practice, the condition like \eqref{cond3} can be difficult to verify. In this paper, the vector $\mathbf e$ will be taken as $\mathbf 1$ consisting of all ones in Theorem \ref{thm2}.

\section{Monotonicity of $Q^3$ spectral element method on a uniform mesh}
\label{sec-q3}
\subsection{The main approach for proving monotonicity}
Even though Lorenz's condition Theorem \ref{thm3} can be nicely verified for the $Q^2$ spectral element method in finite difference form \cite{li2019monotonicity}, it is very difficult to apply Lorenz's condition to higher order $Q^k$ spectral element methods due to their much more complicated structure. In particular, even for $Q^3$ scheme,
simple decomposition of $A_a^- =A^z + A^s$ such that  $A_a^+ \leq A^z  A_d^{-1}A^s $ is difficult to show. Instead, we propose to apply Theorem \ref{thm3} to a few simpler intermediate and auxiliary matrices, then use Theorem \ref{thm2}. To be specific, let $A = A_3$ be the matrix representation of $Q^3$ spectral element method in finite difference form, and let $A_0 = M_1$ be an M-matrix. Then we seek to construct matrices $A_i$ and $L_i$ satisfying the conditions in Theorem \ref{thm3} such that 
\[A_1\leq A_0 L_0, \quad A_2\leq A_1 L_1, \quad A_3\leq A_2 L_2,\]
with the constraints that $A_i \mathbf 1\geq 0$ and $A_0 = M_1$ connects $\mathcal{N}^0(A_i\mathbf{1})$ with $\mathcal{N}^+(A_i\mathbf{1})$ for all $A_i$.
With $\mathbf e=\mathbf 1$ in Theorem \ref{thm2}, we have 
$$A_1 \leq A_0L_0 = M_1L_0 \Rightarrow{} A_1 = M_1M_2 \Rightarrow{} A_2 \leq M_1M_2L_1 \Rightarrow{} A_2 = M_1M_2M_3 $$
$$\Rightarrow A_3 \leq M_1M_2M_3L_2 \Rightarrow A = A_3 = M_1M_2M_3M_4.$$

We remark that the matrices $A_i$ and $L_i$ satisfying constraints above may not be unique. 
It is tedious to verify  the inequalities for the matrices listed in the rest of the section, especially for the matrices for the two-dimensional scheme.
 {\it For reviewers' convenience, we provide a symbolic computation code in MATLAB for easily verifying these inequalities. They can be also downloaded at}
\href{https://www.math.purdue.edu/~zhan1966/research/code/Q3proof.tar.gz}{https://www.math.purdue.edu/$\sim$ zhan1966/research/code/Q3proof.tar.gz}.
{\bf We emphasize that  the proof in this section is a rigorous constructive mathematical proof, with or without such a code for verifying it. }
\subsection{One-dimensional scheme}

We first demonstrate the main idea for the one-dimensional case, for which we only need to construct matrices such that $A_1\leq A_0 L_0,  A\leq A_1 L_1.$

Let $\bar L_h$ denote the coefficient matrix in \eqref{q3-scheme-1d}, then consider $A=\frac{h^2}{4}\bar L_h$.
For convenience, we will perceive the matrix $A$ as a linear operator $\mathcal A$.
Notice that the coefficients for two interior points are symmetric in \eqref{q3-scheme-1d}, thus we will only show stencil for the left interior point for simplicity:
\begin{align*}
&\text{ $\mathcal A$ at boundary   point $x_0$ or $x_{n+1}$}:
      \mathbf{\frac{h^2}{4}} \\
    &\text{ $\mathcal A$ at knot}:
       \Adkth  \; \; \; \; \Adkt \; \; \; \; \Adko \; \; \; \; \Adkz \; \; \; \; \Adko \; \; \; \; \Adkt \; \; \; \; \Adkth \\
    &\text{ $\mathcal A$ at interior point}:\; \; \; \;
       \Adise  \; \; \; \;  \Adiz \; \; \; \; \Adis \; \; \; \; \Adif,
\end{align*}
where bolded entries indicate the coefficient for the operator output location $x_i$.  

For all the matrices    defined below, they will have symmetric structure at two interior points, thus for simplicity we will only show the
stencil of the corresponding linear operators for the left interior point. 
We first define three matrices $A_1$,  $A_0$, and $Z_0$.
\begin{align*}
\text{$\mathcal A_1$ at boundary}: \; \; \; \; &  \mathbf{\frac{h^2}{4}}\\
    \text{$\mathcal A_1$ at knot}: \; \; \; \;
   &    \Aodkth  \; \; \; \; \Aodkt \; \; \; \; \Aodko \; \; \; \; \Aodkz \; \; \; \; \Aodko \; \; \; \; \Aodkt \; \; \; \; \Aodkth \\
    \text{$\mathcal A_1$ at interior point:} \; \; \; \;
    &   \Aodise  \; \; \; \; \Aodiz \; \; \; \; \Aodis \; \; \; \; \Aodif \\
    \text{$\mathcal A_0$ at boundary}: \; \; \; \; &  \mathbf{\frac{h^2}{4}}\\
    \text{$\mathcal A_0$ at knot:} \; \; \; \;
 &      \Modkth \; \; \; \; \Modkt \; \; \; \; \Modko \; \; \; \; \Modkz \; \; \; \; \Modko \; \; \; \; \Modkt \; \; \; \; \Modkth \\
    \text{$\mathcal A_0$ at interior point:} \; \; \; \;
   &    \Modise  \; \; \; \;  \Modiz \; \; \; \; \Modis \; \; \; \; \Modif \\
   \text{$\mathcal Z_0$ at boundary}: \; \; \; \; &\mathbf{0} \\
    \text{$\mathcal Z_0$ at knot:} \; \; \; \;
      & \Asodkth  \; \; \; \; \Asodkt \; \; \; \; \Asodko \; \; \; \; \Asodkz \; \; \; \; \Asodko \; \; \; \; \Asodkt \; \; \; \; \Asodkth \\
    \text{$\mathcal Z_0$ at interior point:} \; \; \; \;
      & \Asodise  \; \; \; \;  \Asodiz \; \; \; \; \Asodis \; \; \; \; \Asodif  \end{align*}
Then we define $L_0= I+(A_0)_d^{-1}Z_0$ where $I$ is the identity matrix and $(A_0)_d$ denotes the diagonal part of $A_0$. By considering composition of two operators $\mathcal A_0$ and $\mathcal L_0$, we get the matrix product $A_0L_0$. 
 Due to the definition of $Z_0$, $\mathcal A_0\mathcal L_0$ still has the same stencil as above:      
      \begin{align*}
        \text{$\mathcal A_0 \mathcal L_0$ at boundary}: \; \; \; \; &\mathbf{\frac{h^2}{4}} \\
    \text{ $\mathcal A_0\mathcal L_0$ at knot:} \; \; \; \;
 &      \MoLodkth  \; \; \; \; \MoLodkt \; \; \; \; \MoLodko \; \; \; \; \MoLodkz \; \; \; \; \MoLodko \; \; \; \; \MoLodkt \; \; \; \; \MoLodkth \\
    \text{$\mathcal A_0\mathcal L_0$ at interior point:} \; \; \; \;
&       \MoLodise \; \; \; \;  \MoLodiz \; \; \; \; \MoLodis \; \; \; \; \MoLodif
\end{align*}

It is straightforward to see $A_1\leq A_0 L_0$. By Theorem \ref{rowsumcondition-thm}, $A_0$ is an M-matrix, thus we set $M_1=A_0$. 
Also it is easy to see that $\mathcal A_1(\mathbf  1)> 0$ thus $\mathcal N^0(A_1\mathbf 1)$ is an empty set. So
$A_0$ trivially connects $\mathcal N^0(A_1\mathbf 1)$ with $\mathcal N^+(A_1\mathbf 1)$. 
By Theorem \ref{thm2},  we have  $A_1 \leq A_0L_0 = M_1L_0 \Rightarrow{} A_1 = M_1M_2$ where $M_2$ is an M-matrix. 

Let $ (A_1)_d$ denote the diagonal part of $A_1$. Then define $L_1 = I + (A_1)_d^{-1}Z_1$ using the following $Z_1$:
\begin{align*}
 \text{$\mathcal Z_1$ at boundary:} \; \; \; \;
&  \mathbf 0\\
    \text{$\mathcal Z_1$ at knot:} \; \; \; \;
&       \Astdkth  \; \; \; \; \Astdkt \; \; \; \; \Astdko \; \; \; \; \Astdkz \; \; \; \; \Astdko \; \; \; \; \Astdkt \; \; \; \; \Astdkth \\
    \text{$\mathcal Z_1$ at interior point:} \; \; \; \;
  &     \Astdise  \; \; \; \;  \Astdiz \; \; \; \; \Astdis \; \; \; \; \Astdif
\end{align*}
And the matrix  $A_1L_1$ still have the same stencil and symmetry:
\begin{align*}
& \text{$\mathcal A_1\mathcal L_1$ at boundary:} \; \; \; \;
  \mathbf{\frac{h^2}{4}}\\
    &\resizebox{1\hsize}{0.020\textheight}{
    \text{$\mathcal A_1\mathcal L_1$ at knot:} 
       $\AoLtdkth$  \; \; $\AoLtdkt$ \; \;  $\AoLtdko$ \; \; $\AoLtdkz$ \; \; $\AoLtdko$ \; \; $\AoLtdkt$ \; \;$\AoLtdkth$}\\
  &  \text{$\mathcal A_1\mathcal L_1$ at  interior point:} \; \; \; \;
       \AoLtdise \; \; \; \;  \AoLtdiz \; \; \; \; \AoLtdis \; \; \; \; \AoLtdif
\end{align*}
A direct comparison  verifies that $A \leq A_1L_1=M_1M_2L_1$.
Also it is easy to see that $\mathcal A(\mathbf  1)_i=0$ if $x_i$ is not a boundary point.
The operator $\mathcal A_0$ has a three-point stencil at interior grid points, thus the directed graph defined by the adjacency matrix $A_0$ has a directed path starting from
any interior grid point to any other point, see Figure \ref{q3_graph_1D}. 
So
$M_1=A_0$  connects $\mathcal N^0(A \mathbf 1)$ with $\mathcal N^+(A \mathbf 1)$. 
By Theorem \ref{thm2},  we have  $A \leq A_1L_1=M_1M_2L_1 \Rightarrow{} A = M_1M_2 M_3$ where $M_3$ is an M-matrix. 
Therefore, $A^{-1}=M_3^{-1}M_2^{-1} M_1^{-1}\geq 0.$

\begin{figure}[htbp]
    \label{q3_graph_1D} \centering \includegraphics[scale=1.4]{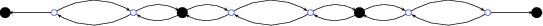}
    \caption{The directed graph defined by matrix $M_1$ for the finite difference grid shown in Figure \ref{q3_stencil_def_1D}.}
\end{figure}

\subsection{Two-dimensional case}
\begin{figure}[htbp]
\subfigure[Three point types defining the stencil: knot (black), edge point (blue), interior point (green).]{\includegraphics[scale=0.9]{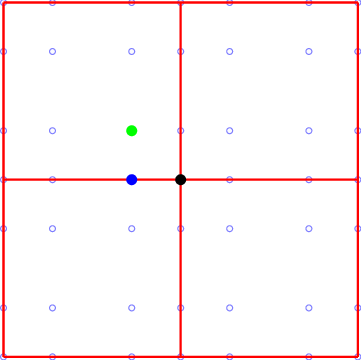}}
\subfigure[The directed graph defined by the matrix $M_1$.]{\includegraphics[scale=0.9]{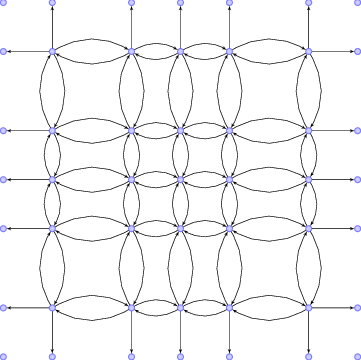}}
\hspace{0.3in}
    \caption{An illustration of a $Q^3$ mesh with $2\times 2$ cells.} 
    \label{q3_pointtype}
\end{figure}
Due to symmetry, the stencil of the scheme can be defined at three different types of points, 
see Figure \ref{q3_pointtype} (a). Let each rectangular cell have size $h \times h$ and denote  $Q^3$ scheme  by $\bar L_h\bar {\mathbf u} =\bar {\mathbf f} $. Let $A=\frac{h^2}{4}\bar L_h$.
Then for a boundary point $(x_i, y_j)\in \partial\Omega$, $\mathcal A(\bar {\mathbf u})_{ij}=\frac{h^2}{4}u_{ij}$. And the stencil of $\mathcal A$ at interior grid points is given as

\begin{table}[H]
\centering    
    \resizebox{1.1\textwidth}{!}{%
    \begin{tabular}{lr}
    \multicolumn{2}{c}{
    \begin{tabular}{rccccccc} 
        &   &   &   &   \Addkth   &   &   &   
        \\
        \\
        &   &   &   &   \Addkt   &   &   &   
        \\
        \\
        &   &   &   &  \Addko   &   &   &   
        \\
        \\
      \bf{$\mathcal A$ at knot}:  & \Addkth & \Addkt  & \Addko  &  \Addkz   & \Addko  &  \Addkt &  \Addkth
        \\
        \\
        &   &   &   &   \Addko &   &   &   
        \\
        \\
        &   &   &   &   \Addkt   &   &   &   
        \\
        \\
        &   &   &   & \Addkth   &   &   &   
    \end{tabular}
    
    }
    \\
    \begin{tabular}{rcccc} 

        &   &   &   \Addeth   &     
        \\
        \\
        &   &   &   \Addet   &    
        \\
        \\
        &   &   &  \Addeo   &    
        \\
        \\
      \bf{$\mathcal A$ at edge point}:  & \Addef & \Addes   &  \Addez & \Addese
        \\
        \\
        &   &   &   \Addeo &     
        \\
        \\
        &   &   &   \Addet   &     
        \\
        \\
        &   &   & \Addeth   &     
    \end{tabular}

    &

    \begin{tabular}{rcccc} 
        &   &   &   \Addif   &    
        \\
        \\
        &   &   &  \Addis   &    
        \\
        \\
      \bf{$\mathcal A$ at interior point:}  & \Addif & \Addis   &  \Addiz & \Addise
        \\
        \\
        &   &   &   \Addise & 
        \\
        \\
        &   &   &   & 
    \end{tabular}
    \end{tabular}}
      
\end{table}

Next we list the definition of matrices $A_i$ and $Z_i$ by the corresponding linear operators $\mathcal A_i$ and $ \mathcal  Z_i$. For convenience, we will only list the stencil at interior grid points. For the domain boundary points $(x_i, y_j)\in\partial\Omega$, all $A_i$ matrices will have the same value as $A$: $\mathcal A_i(\bar{\mathbf u})_{ij}=\frac{h^2}{4}u_{ij}.$ And  $\mathcal Z_i(\bar{\mathbf u})_{ij}=0$ for  $(x_i, y_j)\in\partial\Omega$. 
The matrix $L_i$ is defined as 
$L_i = I + (A_{i})_d^{-1}Z_i,\quad i=0,1,2.$ 
The matrices and their products are given by:

\begin{table}[H]
\centering
    
    \resizebox{ 1.2\textwidth}{!}{%
    \begin{tabular}{lr}
    \multicolumn{2}{c}{
    \begin{tabular}{rccccccc} 
        &   &   &   &   \Aoddkth   &   &   &   
        \\
        \\
        &   &   &   &   \Aoddkt   &   &   &   
        \\
        \\
        &   &   &   &  \Aoddko   &   &   &   
        \\
        \\
      \bf{$\mathcal A_1$ at knot}:  & \Aoddkth & \Aoddkt  & \Aoddko  &  \Aoddkz   & \Aoddko  &  \Aoddkt &  \Aoddkth
        \\
        \\
        &   &   &   &   \Aoddko &   &   &   
        \\
        \\
        &   &   &   &   \Aoddkt   &   &   &   
        \\
        \\
        &   &   &   & \Aoddkth   &   &   &   
    \end{tabular}
    
    }
    
    \\
    \begin{tabular}{rcccc} 

        &   &   &   \Aoddeth   &     
        \\
        \\
        &   &   &   \Aoddet   &    
        \\
        \\
        &   &   &  \Aoddeo   &    
        \\
        \\
      \bf{$\mathcal A_1$ at edge point}:  & \Aoddef & \Aoddes   &  \Aoddez & \Aoddese
        \\
        \\
        &   &   &   \Aoddeo &     
        \\
        \\
        &   &   &   \Aoddet   &     
        \\
        \\
        &   &   & \Aoddeth   &     
    \end{tabular}

    &

    \begin{tabular}{rcccc} 
        &   &   &   \Aoddif   &    
        \\
        \\
        &   &   &  \Aoddis   &    
        \\
        \\
      \bf{$\mathcal A_1$ at interior point:}  & \Aoddif & \Aoddis   &  \Aoddiz & \Aoddise
        \\
        \\
        &   &   &   \Aoddise & 
        \\
        \\
        &   &   &   & 
    \end{tabular}
    \\
    
     \begin{tabular}{rccccccc} 
        &   &   &   &   \Moddkth   &   &   &   
        \\
        \\
        &   &   &   &   \Moddkt   &   &   &   
        \\
        \\
        &   &   &   &  \Moddko   &   &   &   
        \\
        \\
      \bf{$\mathcal A_0$ at knot}:  & \Moddkth & \Moddkt  & \Moddko  &  \Moddkz   & \Moddko  &  \Moddkt &  \Moddkth
        \\
        \\
        &   &   &   &   \Moddko &   &   &   
        \\
        \\
        &   &   &   &   \Moddkt   &   &   &   
        \\
        \\
        &   &   &   & \Moddkth   &   &   &   
    \end{tabular}\\
    
    \begin{tabular}{rcccc} 

        &   &   &   \Moddeth   &     
        \\
        \\
        &   &   &   \Moddet   &    
        \\
        \\
        &   &   &  \Moddeo   &    
        \\
        \\
      \bf{$\mathcal A_0$ at  edge point}:  & \Moddef & \Moddes   &  \Moddez & \Moddese
        \\
        \\
        &   &   &   \Moddeo &     
        \\
        \\
        &   &   &   \Moddet   &     
        \\
        \\
        &   &   & \Moddeth   &     
    \end{tabular}

    &

    \begin{tabular}{rcccc} 
        &   &   &   \Moddif   &    
        \\
        \\
        &   &   &  \Moddis   &    
        \\
        \\
      \bf{$\mathcal A_0$ at  interior point:}  & \Moddif & \Moddis   &  \Moddiz & \Moddise
        \\
        \\
        &   &   &   \Moddise & 
        \\
        \\
        &   &   &   & 
    \end{tabular} 
    \end{tabular}}

\end{table}

\begin{table}[H]
\centering

    \resizebox{1 \textwidth}{!}{%
    \begin{tabular}{lr}
    \multicolumn{2}{c}{
    \begin{tabular}{rccccccc} 
        &   &   &   &   \Asoddkth   &   &   &   
        \\
        \\
        &   &   &   &   \Asoddkt   &   &   &   
        \\
        \\
        &   &   &   &  \Asoddko   &   &   &   
        \\
        \\
      \bf{$\mathcal Z_0$ at  knot }:  & \Asoddkth & \Asoddkt  & \Asoddko  &  \Asoddkz   & \Asoddko  &  \Asoddkt &  \Asoddkth
        \\
        \\
        &   &   &   &   \Asoddko &   &   &   
        \\
        \\
        &   &   &   &   \Asoddkt   &   &   &   
        \\
        \\
        &   &   &   & \Asoddkth   &   &   &   
    \end{tabular}
    
    }
    
    \\
    \begin{tabular}{rcccc} 

        &   &   &   \Asoddeth   &     
        \\
        \\
        &   &   &   \Asoddet   &    
        \\
        \\
        &   &   &  \Asoddeo   &    
        \\
        \\
      \bf{$\mathcal Z_0$ at  edge point}:  & \Asoddef & \Asoddes   &  \Asoddez & \Asoddese
        \\
        \\
        &   &   &   \Asoddeo &     
        \\
        \\
        &   &   &   \Asoddet   &     
        \\
        \\
        &   &   & \Asoddeth   &     
    \end{tabular}

    &

    \begin{tabular}{rcccc} 
        &   &   &   \Asoddif   &    
        \\
        \\
        &   &   &  \Asoddis   &    
        \\
        \\
      \bf{$\mathcal Z_0$ at  interior point:}  & \Asoddif & \Asoddis   &  \Asoddiz & \Asoddise
        \\
        \\
        &   &   &   \Asoddise & 
        \\
        \\
        &   &   &   & 
    \end{tabular}
    \end{tabular}}
\centering
    
    \resizebox{1.2\textwidth}{!}{%
    \begin{tabular}{lr}
    \multicolumn{2}{c}{
    \begin{tabular}{rccccccc} 
        &   &   &   &   \MoLoddkth   &   &   &   
        \\
        \\
        &   &   &   &   \MoLoddkt   &   &   &   
        \\
        \\
        &   &   &   &  \MoLoddko   &   &   &   
        \\
        \\
      \bf{$\mathcal A_0 \mathcal L_0$ at  knot}:  & \MoLoddkth & \MoLoddkt  & \MoLoddko  &  \MoLoddkz   & \MoLoddko  &  \MoLoddkt &  \MoLoddkth
        \\
        \\
        &   &   &   &   \MoLoddko &   &   &   
        \\
        \\
        &   &   &   &   \MoLoddkt   &   &   &   
        \\
        \\
        &   &   &   & \MoLoddkth   &   &   &   
    \end{tabular}
    
    }
    
    \\
    \begin{tabular}{rcccc} 

        &   &   &   \MoLoddeth   &     
        \\
        \\
        &   &   &   \MoLoddet   &    
        \\
        \\
        &   &   &  \MoLoddeo   &    
        \\
        \\
      \bf{$\mathcal A_0 \mathcal L_0$ at  edge point}:  & \MoLoddef & \MoLoddes   &  \MoLoddez & \MoLoddese
        \\
        \\
        &   &   &   \MoLoddeo &     
        \\
        \\
        &   &   &   \MoLoddet   &     
        \\
        \\
        &   &   & \MoLoddeth   &     
    \end{tabular}

    &

    \begin{tabular}{rcccc} 
        &   &   &   \MoLoddif   &    
        \\
        \\
        &   &   &  \MoLoddis   & \MoLoddie   
        \\
        \\
      \bf{$\mathcal A_0 \mathcal L_0$ at  interior point:}  & \MoLoddif & \MoLoddis   &  \MoLoddiz & \MoLoddise
        \\
        \\
        &   & \MoLoddie  &   \MoLoddise & 
        \\
        \\
        &   &   &   & 
    \end{tabular}
    \end{tabular}}

\end{table}

\begin{table}[H]
\centering
    
    \resizebox{\textwidth}{!}{%
    \begin{tabular}{lr}
    \multicolumn{2}{c}{
    \begin{tabular}{rccccccc} 
        &   &   &   &   \Atddkth   &   &   &   
        \\
        \\
        &   &   &   &   \Atddkt   &   &   &   
        \\
        \\
        &   &   &   &  \Atddko   &   &   &   
        \\
        \\
      \bf{$\mathcal A_2$ at  knot}:  & \Atddkth & \Atddkt  & \Atddko  &  \Atddkz   & \Atddko  &  \Atddkt &  \Atddkth
        \\
        \\
        &   &   &   &   \Atddko &   &   &   
        \\
        \\
        &   &   &   &   \Atddkt   &   &   &   
        \\
        \\
        &   &   &   & \Atddkth   &   &   &   
    \end{tabular}
    
    }
    
    \\
    \begin{tabular}{rcccc} 

        &   &   &   \Atddeth   &     
        \\
        \\
        &   &   &   \Atddet   &  \Atdden  
        \\
        \\
        &   &   &  \Atddeo   &    
        \\
        \\
      \bf{$\mathcal A_2$ at   edge point}:  & \Atddef & \Atddes   &  \Atddez & \Atddese
        \\
        \\
        &   &   &   \Atddeo &     
        \\
        \\
        &   &   &   \Atddet   &  \Atdden   
        \\
        \\
        &   &   & \Atddeth   &     
    \end{tabular}

    &

    \begin{tabular}{rcccc} 
        &   &   &   \Atddif   &    
        \\
        \\
        &   &   &  \Atddis   &  \Atddie  
        \\
        \\
      \bf{$\mathcal A_2$ at   interior point:}  & \Atddif & \Atddis   &  \Atddiz & \Atddise
        \\
        \\
        &   &  \Atddie &   \Atddise & 
        \\
        \\
        &   &   &   & 
    \end{tabular}
    \end{tabular}}
 
\centering
    
    \resizebox{\tabletextwidtho\textwidth}{!}{%
    \begin{tabular}{lr}
    \multicolumn{2}{c}{
    \begin{tabular}{rccccccc} 
        &   &   &   &   \Astddkth   &   &   &   
        \\
        \\
        &   &   &   &   \Astddkt   &   &   &   
        \\
        \\
        &   &   &   &  \Astddko   &   &   &   
        \\
        \\
      \bf{$\mathcal Z_1 $ at   knot }:  & \Astddkth & \Astddkt  & \Astddko  &  \Astddkz   & \Astddko  &  \Astddkt &  \Astddkth
        \\
        \\
        &   &   &   &   \Astddko &   &   &   
        \\
        \\
        &   &   &   &   \Astddkt   &   &   &   
        \\
        \\
        &   &   &   & \Astddkth   &   &   &   
    \end{tabular}
    
    }
    
    \\
    \begin{tabular}{rcccc} 

        &   &   &   \Astddeth   &     
        \\
        \\
        &   &   &   \Astddet   &    
        \\
        \\
        &   &   &  \Astddeo   &    
        \\
        \\
      \bf{$\mathcal Z_1 $ at   edge point}:  & \Astddef & \Astddes   &  \Astddez & \Astddese
        \\
        \\
        &   &   &   \Astddeo &     
        \\
        \\
        &   &   &   \Astddet   &     
        \\
        \\
        &   &   & \Astddeth   &     
    \end{tabular}

    &

    \begin{tabular}{rcccc} 
        &   &   &   \Astddif   &    
        \\
        \\
        &   &   &  \Astddis   &  \Astddie  
        \\
        \\
      \bf{$\mathcal Z_1 $ at   interior point:}  & \Astddif & \Astddis   &  \Astddiz & \Astddise
        \\
        \\
        &   &  \Astddie &   \Astddise & 
        \\
        \\
        &   &   &   & 
    \end{tabular}
    \end{tabular}}
 
\centering
    
    \resizebox{1.2\textwidth}{!}{%
    \begin{tabular}{lr}
    \multicolumn{2}{c}{
    \begin{tabular}{rccccccc} 
        &   &   &   &   \AoLtddkth   &   &   &   
        \\
        \\
        &   &   &   &   \AoLtddkt   &   &   &   
        \\
        \\
        &   &   &   &  \AoLtddko   &   &   &   
        \\
        \\
      \bf{$\mathcal A_1\mathcal L_1$ at knot}:  & \AoLtddkth & \AoLtddkt  & \AoLtddko  &  \AoLtddkz   & \AoLtddko  &  \AoLtddkt &  \AoLtddkth
        \\
        \\
        &   &   &   &   \AoLtddko &   &   &   
        \\
        \\
        &   &   &   &   \AoLtddkt   &   &   &   
        \\
        \\
        &   &   &   & \AoLtddkth   &   &   &   
    \end{tabular}
    
    }
    
    \\
    \begin{tabular}{rcccc} 

        &   &   &   \AoLtddeth   &     
        \\
        \\
        &   &   &   \AoLtddet   &  \AoLtdden   
        \\
        \\
        &   &  \AoLtddet &  \AoLtddeo   &    
        \\
        \\
      \bf{$\mathcal A_1\mathcal L_1$ at edge point}:  & \AoLtddef & \AoLtddes   &  \AoLtddez & \AoLtddese
        \\
        \\
        &   & \AoLtddet  &   \AoLtddeo &     
        \\
        \\
        &   &   &   \AoLtddet   &   \AoLtdden  
        \\
        \\
        &   &   & \AoLtddeth   &     
    \end{tabular}

    &

    \begin{tabular}{rcccc} 
        &   &  \AoLtddiel &   \AoLtddif   &    
        \\
        \\
        &  \AoLtddiel & \AoLtddite  &  \AoLtddis   &    \AoLtddie
        \\
        \\
      \bf{$\mathcal A_1\mathcal L_1$ interior point:}  & \AoLtddif & \AoLtddis   &  \AoLtddiz & \AoLtddise
        \\
        \\
        &   &  \AoLtddie &   \AoLtddise & \AoLtdditw
        \\
        \\
        &   &   &   & 
    \end{tabular}
    \end{tabular}}

\end{table}

\begin{table}[H]
\centering
    
    \resizebox{\tabletextwidtho\textwidth}{!}{%
    \begin{tabular}{lr}
    \multicolumn{2}{c}{
    \begin{tabular}{rccccccc} 
        &   &   &   &   \Asthddkth   &   &   &   
        \\
        \\
        &   &   &   &   \Asthddkt   &   &   &   
        \\
        \\
        &   &   &   &  \Asthddko   &   &   &   
        \\
        \\
      \bf{$\mathcal Z_2$ at knot}:  & \Asthddkth & \Asthddkt  & \Asthddko  &  \Asthddkz   & \Asthddko  &  \Asthddkt &  \Asthddkth
        \\
        \\
        &   &   &   &   \Asthddko &   &   &   
        \\
        \\
        &   &   &   &   \Asthddkt   &   &   &   
        \\
        \\
        &   &   &   & \Asthddkth   &   &   &   
    \end{tabular}
    
    }
    
    \\
    \begin{tabular}{rcccc} 

        &   &   &   \Asthddeth   &     
        \\
        \\
        &   &   &   \Asthddet   &    
        \\
        \\
        &   &   &  \Asthddeo   &    
        \\
        \\
      \bf{$\mathcal Z_2$ at edge point}:  & \Asthddef & \Asthddes   &  \Asthddez & \Asthddese
        \\
        \\
        &   &   &   \Asthddeo &     
        \\
        \\
        &   &   &   \Asthddet   &     
        \\
        \\
        &   &   & \Asthddeth   &     
    \end{tabular}

    &

    \begin{tabular}{rcccc} 
        &   &   &   \Asthddif   &    
        \\
        \\
        &   &   &  \Asthddis   &    
        \\
        \\
      \bf{$\mathcal Z_2$ at interior point:}  & \Asthddif & \Asthddis   &  \Asthddiz & \Asthddise
        \\
        \\
        &   &   &   \Asthddise & 
        \\
        \\
        &   &   &   & 
    \end{tabular}
    \end{tabular}}

\centering
    
    \resizebox{1.2\textwidth}{!}{%
    \begin{tabular}{lr}
    \multicolumn{2}{c}{
    \begin{tabular}{rccccccc} 
        &   &   &   &   \AtLthddkth   &   &   &   
        \\
        \\
        &   &   &   &   \AtLthddkt   &   &   &   
        \\
        \\
        &   &   &   &  \AtLthddko   &   &   &   
        \\
        \\
      \bf{$\mathcal A_2\mathcal L_2$ at knot }:  & \AtLthddkth & \AtLthddkt  & \AtLthddko  &  \AtLthddkz   & \AtLthddko  &  \AtLthddkt &  \AtLthddkth
        \\
        \\
        &   &   &   &   \AtLthddko &   &   &   
        \\
        \\
        &   &   &   &   \AtLthddkt   &   &   &   
        \\
        \\
        &   &   &   & \AtLthddkth   &   &   &   
    \end{tabular}
    
    }
    
    \\
    \hline
    \begin{tabular}{rcccc} 

        &   &   &   \AtLthddeth   &     
        \\
        \\
        &   &   &   \AtLthddet   &  \AtLthdden  
        \\
        \\
        &   &   &  \AtLthddeo   &    \AtLthddethirt
        \\
        \\
      \bf{$\mathcal A_2\mathcal L_2$ at edge point}:  & \AtLthddef & \AtLthddes   &  \AtLthddez & \AtLthddese
        \\
        \\
        &   &   &   \AtLthddeo &    \AtLthddethirt 
        \\
        \\
        &   &   &   \AtLthddet   &   \AtLthdden  
        \\
        \\
        &   &   & \AtLthddeth   &     
    \end{tabular}

    \\
    
   \hline
    
    \begin{tabular}{rcccc} 
        &   & \AtLthddiel  &   \AtLthddif   &    
        \\
        \\
        &  \AtLthddiel & \AtLthddite  &  \AtLthddis   &    \AtLthddie
        \\
        \\
      \bf{$\mathcal A_2\mathcal L_2$ at interior point:}  & \AtLthddif & \AtLthddis   &  \AtLthddiz & \AtLthddise
        \\
        \\
        &   & \AtLthddie  &   \AtLthddise & \AtLthdditw
        \\
        \\
        &   &   &   & 
    \end{tabular}
    \end{tabular}}

\end{table}

By Theorem \ref{rowsumcondition-thm}, $A_0$ is an M-matrix, thus we set $M_1=A_0$. Notice that the matrix $M_1=A_0$ has a 5-point stencil and the directed graph defined by $M_1$ is given in Figure \ref{q3_pointtype} (b), in which there is a directed path starting from any interior grid point to any other point.
For convenience, let $A_3=A$. 
Then we have $\mathcal A_k(\mathbf 1)\geq 0$ $(k=0,1,2,3)$. Moreover, $\mathcal A_k(\mathbf 1)_{ij}>0$ $(k=0,1,2,3)$ for any domain boundary point $(x_i, y_j)\in\partial\Omega$.  The directed graph  defined by $M_1$   easily implies that $M_1$ connects $\mathcal N^0(A_i\mathbf 1)$ with $\mathcal N^+(A_i\mathbf 1)$ for all $i=0,1,2,3$.

By straightforward comparison, we can verify that $A_1\leq A_0 L_0,  A_2\leq A_1 L_1,  A\leq A_2 L_2$. 
By Theorem \ref{thm2},  we have 
$$A_1 \leq A_0L_0 = M_1L_0 \Rightarrow{} A_1 = M_1M_2 \Rightarrow{} A_2 \leq M_1M_2L_1 \Rightarrow{} A_2 = M_1M_2M_3 $$
$$\Rightarrow A \leq M_1M_2M_3L_2 \Rightarrow A = M_1M_2M_3M_4 \Rightarrow A^{-1}\geq 0.$$

\begin{rmk}
 The matrices $A_i$ and $L_i$ are found by matching the inequalities above. Such matrices are not unique. 
 These matrices and the inequalities can be easily verified by computer codes.
 
\end{rmk}

\section{Numerical Tests}
\label{sec:5}
\label{sec-test}
\subsection{Monotonicity tests}

For solving a one-dimensional Poisson equation $-u''=f$ on the domain $(0,1)$ with homogeneous Dirichlet boundary condition $u(0)=u(1)=0$,
consider the classical continuous finite element method using $P^k$ polynomial basis on a uniform  mesh consisting of $N$ intervals.
If all the integrals are replaced by $(k+1)$-point Gauss-Lobatto quadrature, then it is equivalent to a finite difference scheme at all Gauss-Lobatto points excluding two domain boundary points. The finite difference scheme can be written as $S\mathbf u=M\mathbf f$ or $Hu=\mathbf f$, where $S$ is the stiffness matrix, $M\geq 0$ is the lumped mass matrix
and $H=M^{-1}S$. The results in \cite{vejchodsky2007discrete} imply that $S\geq 0$ thus $H\geq 0$ for any $k$.
See Figure \ref{1dmono} for the smallest entry in the matrix $H^{-1}$ for $k=2,3,\dots, 15$ on different meshes.

For a two-dimensional Poisson equation $-u_{xx}-u_{yy}=f$ on the domain $(0,1)\times(0,1)$ with homogeneous Dirichlet boundary condition, 
the $Q^k$ spectral element method, i.e., $Q^k$ finite element method with $(k+1)\times(k+1)$-point Gauss-Lobatto quadrature, is equivalent to a finite difference scheme at all Gauss-Lobatto points excluding all domain boundary points. On a uniform mesh consisting of $N\times N$ rectangular cells, the stiffness matrix and lumped matrix can be written as $S\otimes M+M\otimes S$ and $M\otimes M$ respectively \cite{li2019fourth}. 
So the finite difference scheme matrix in two dimensions can be written as 
$$H2D=(M\otimes M)^{-1} (S\otimes M+M\otimes S)=H\otimes I+I\otimes H,$$
where the matrices $S, M, H$ are the same ones as in the one-dimensional scheme. 
Unfortunately, here the Kronecker product and $H^{-1}$ simply does not imply the inverse positivity of $H2D$. In numerical tests on a $N\times N$  mesh, 
there is a clear cut off at $k=9$. For $Q^k$ spectral element method with $k\geq 9$, the inverse positivity is simply lost in two dimensions even on very coarse meshes, see numerical results in Figure \ref{2dmono}.  

In three dimensions, the   finite difference scheme matrix $(-\Delta_h)$ 
of the $Q^k$ spectral element method can be written as
$$H3D=H\otimes I\otimes I+I\otimes H\otimes I+I\otimes I\otimes H.$$
The numerical tests shown in Figure \eqref{fig-3d} suggest that $H3D$ is no longer unconditionally monotone for 
$Q^k$ spectral element method with $k\geq 4$.

\begin{figure}[htbp]
\includegraphics[scale=0.28]{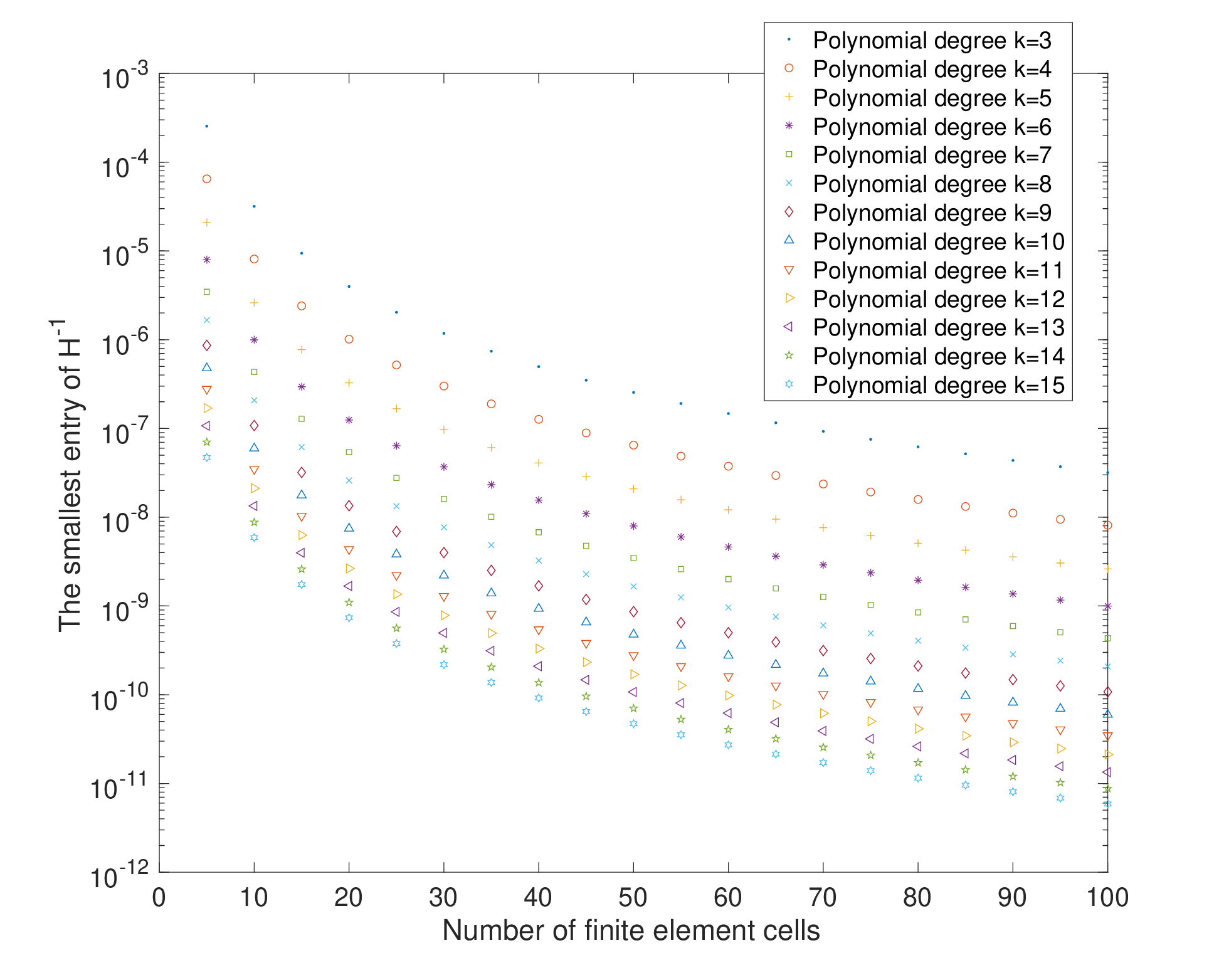}
\hspace{0.3in}
   \caption{$H^{-1}\geq 0$ holds for any $k$ in $P^k$ finite element method with $(k+1)$-point Gauss-Lobatto quadrature on a uniform mesh with $N$ cells for one-dimensional Laplacian.} 
     \label{1dmono}
\end{figure}

\begin{figure}[htbp] 
 \subfigure{\includegraphics[scale=0.25]{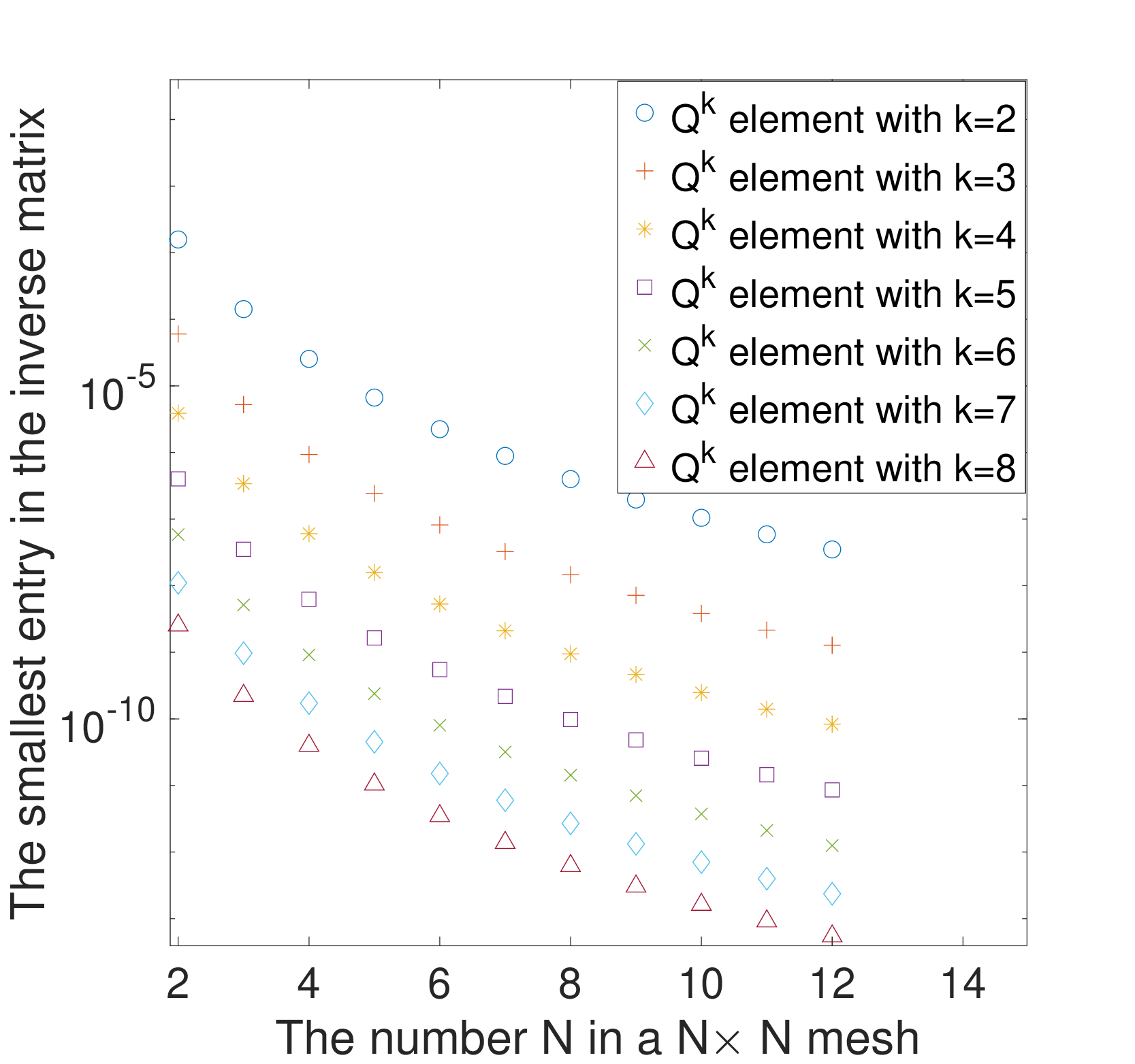} }
 \subfigure{\includegraphics[scale=0.25]{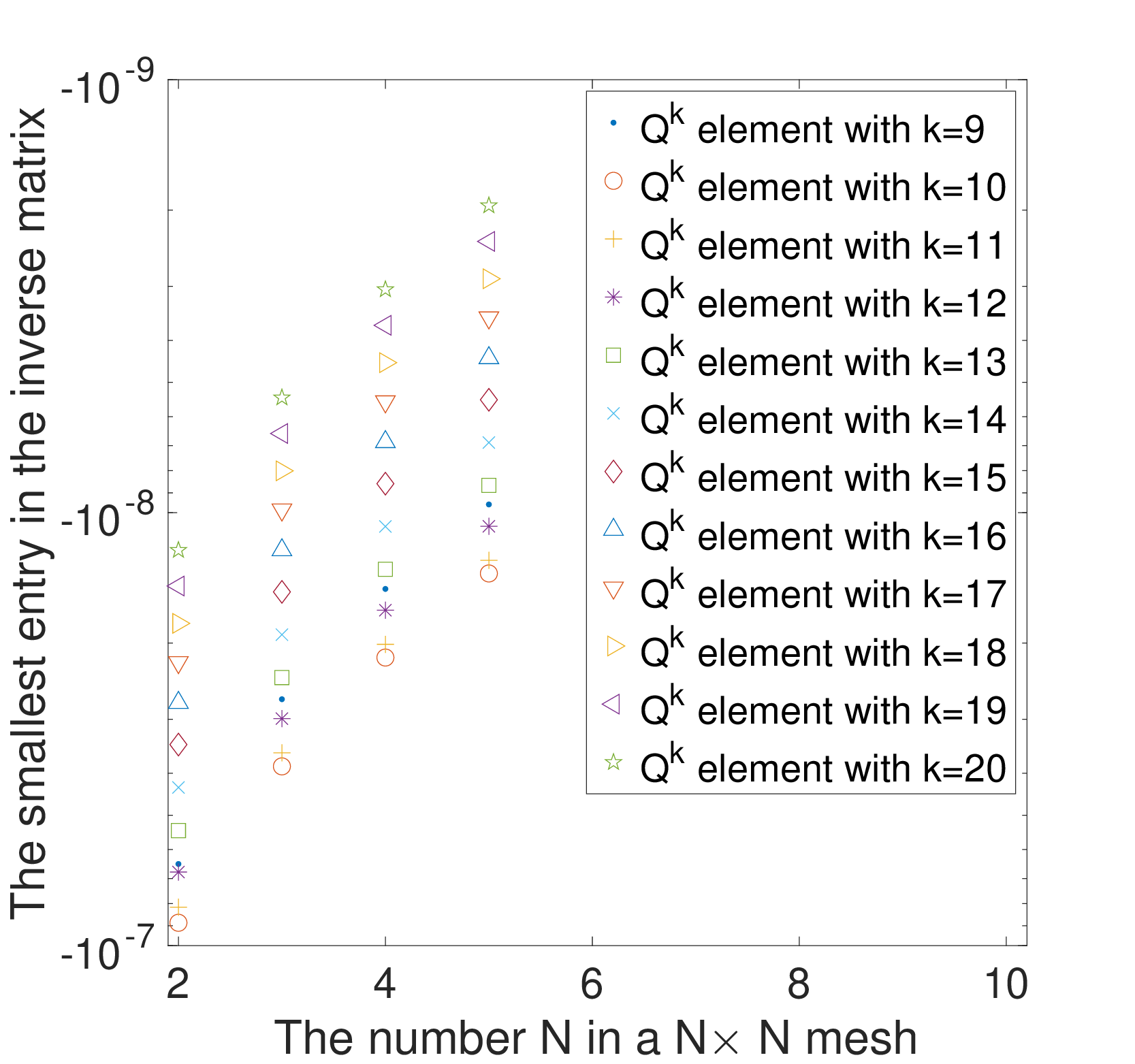} }
\hspace{0.3in}
    \caption{On a uniform mesh with $N\times N$ cells, inverse positivity of the matrix is simply lost in $Q^k$ spectral element method for 2D Laplacian if $k\geq 9$.} 
    \label{2dmono}
\end{figure}

\begin{figure}[htbp]
    \centering 
    \subfigure{\includegraphics[scale=0.28]{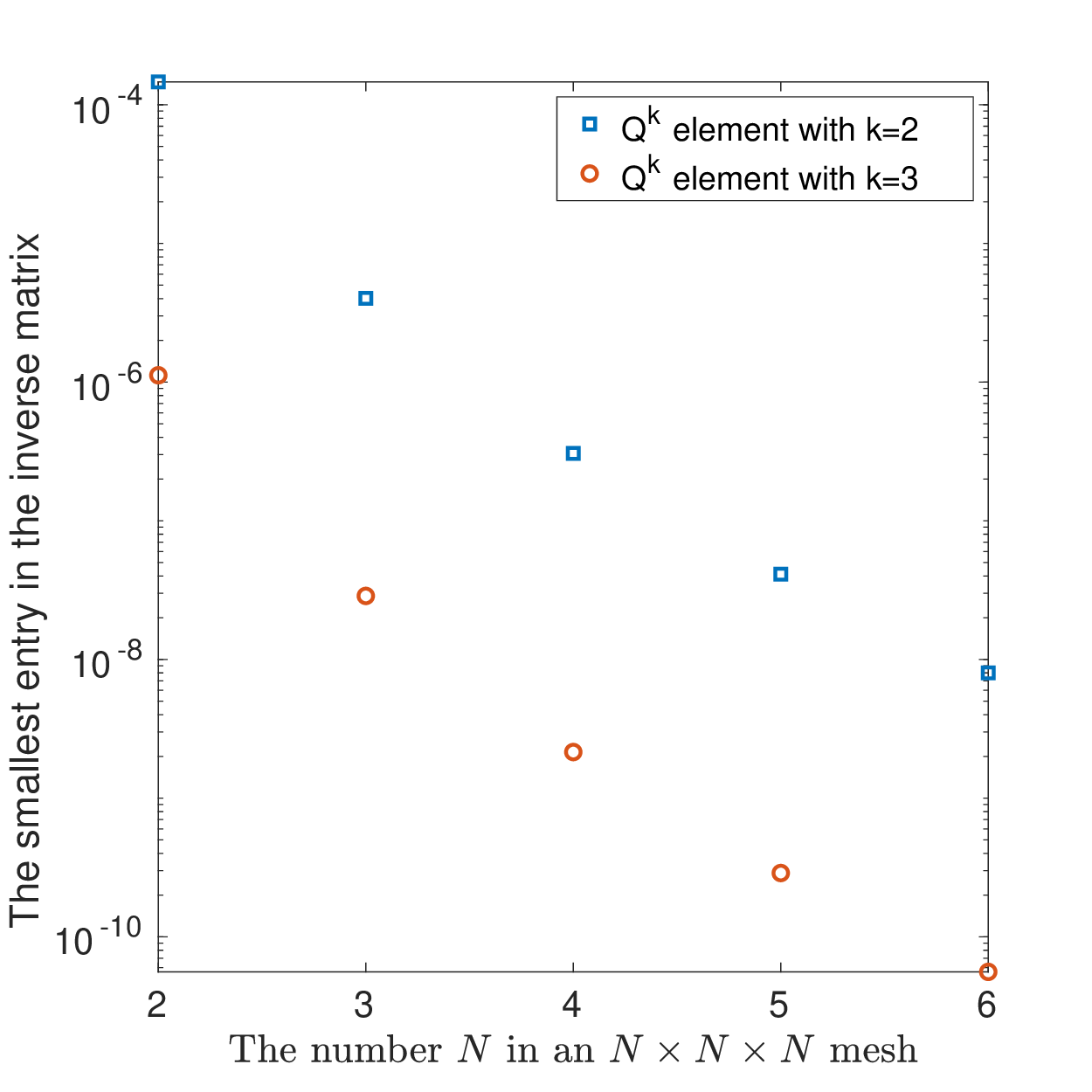}}
    \hspace{-0.5cm}
   \subfigure{ \includegraphics[scale=0.28]{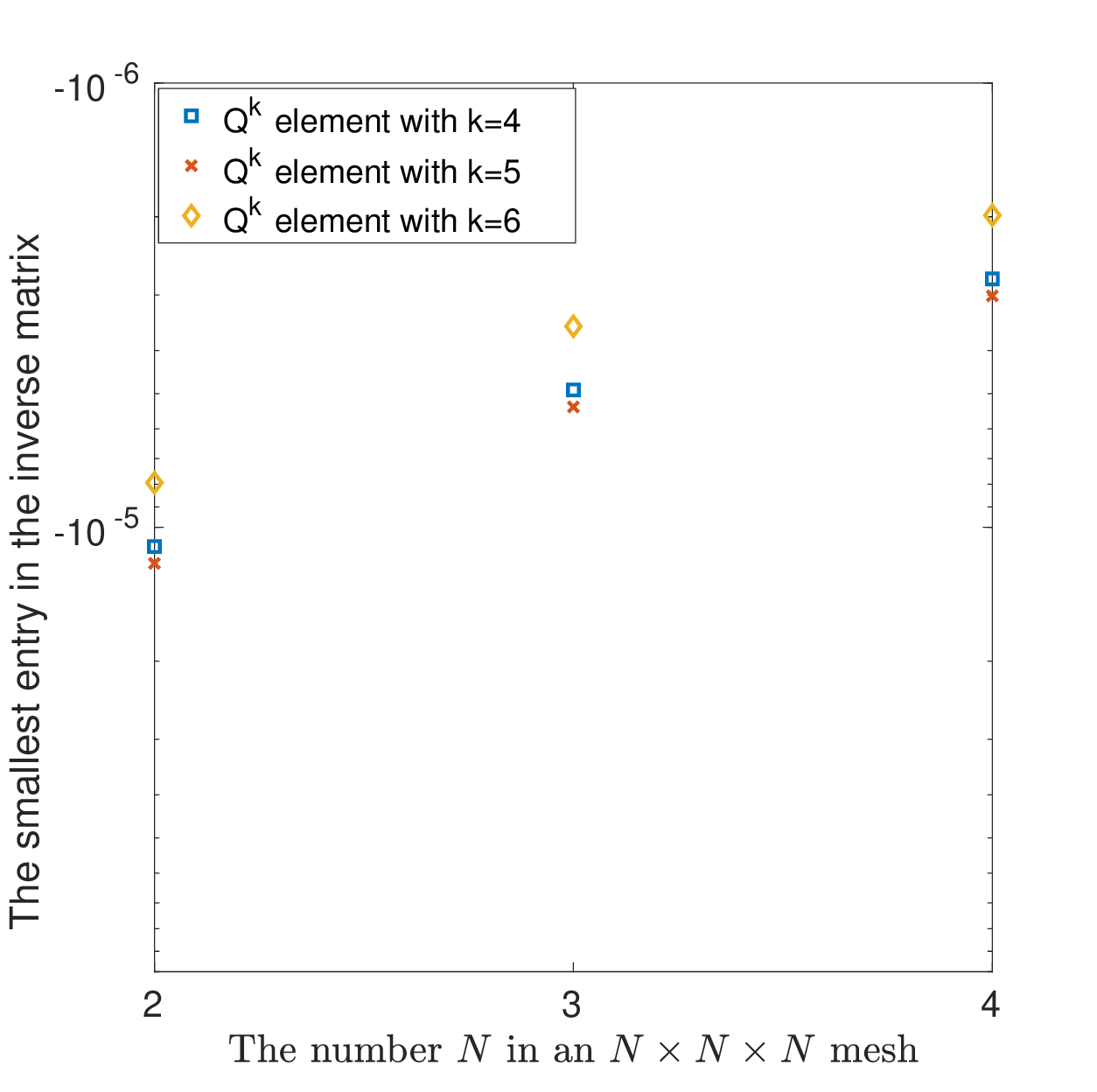} } 
    \caption{The smallest entry in the matrix $(H\otimes I\otimes I+I\otimes H\otimes I+I\otimes I\otimes H)^{-1}$ for $Q^k$ spectral element method on a uniform $N\times N\times N$ mesh for $Q^k$ spectral element method will be negative if $k\geq 4$. }
    \label{fig-3d}
\end{figure}

 \subsection{Accuracy tests}
For verifying the order for smooth solutions, we show some accuracy tests of the monotone schemes mentioned in this paper for solving $-\Delta u=f$ on a square $(0,1)\times{}(0,1)$ with Dirichlet boundary conditions. 
We will simply refer to the classical 9-point scheme \eqref{poisson-scheme3}  as 9-point scheme, and
 refer to its variant \eqref{poisson-scheme2} as compact finite difference.  The schemes are tested for the 
  Poisson equation $-\Delta u=f$ with nonhomogeneous Dirichlet boundary condition:
\begin{equation}\label{numtestpoisson3}
\begin{gathered}
f = 74\pi{}^2cos(5\pi{}x)cos(7\pi{}y) - 8 \\
u = cos(5\pi{}x)cos(7\pi{}y) + x^2 + y^2
\end{gathered}
\end{equation}

The errors of fourth order accurate schemes on uniform grids are listed in Table \ref{table-3}.
The  errors of $Q^3$ spectral element method  on uniform rectangular meshes are listed in  Table \ref{table-q3}.

\begin{table}[htbp]
\centering
\caption{Accuracy test on uniform meshes for \eqref{numtestpoisson3}.}
\renewcommand{\arraystretch}{2}
\resizebox{\textwidth}{!}{
\begin{tabular}{|c|c|c|c|c|c|c|c|c|c|c|c|c|}
\hline
\multirow{2}{*}{Finite Difference Grid} & 
\multicolumn{4}{c|}{$Q^2$ spectral element method} & 
\multicolumn{4}{c|}{$P^2$ finite element method} &
\multicolumn{4}{c|}{9-point scheme \eqref{poisson-scheme3}} \\
\cline{2-13}
 & 
 $l^2$ error & order & $l^\infty$ error & order&
 $l^2$ error & order & $l^\infty$ error & order  &
$l^2$ error & order & $l^\infty$ error & order  \\
\hline
$7\times 7$  & 
        3.62E-1     & - &         1.10E-0     & - & 
        9.68E-1     & - &         2.59E-0     & - & 
        2.48E-2     & - &         5.69E-2     & - 
 \\
\hline
$15\times 15$ & 
        3.75E-2     &     3.26     &         9.68E-2     &     3.50     & 
        7.81E-2     &     3.63     &         3.00E-1     &     3.11     & 
    2.61E-4     &     6.56     &     6.46E-4     &     6.45     
 \\
\hline
$31\times 31$ &
        2.44E-3     &     3.94     &         7.18E-3     &     3.75     & 
        4.70E-3    &     4.05     &         1.84E-2     &     4.02     & 
    3.65E-5     &     2.84     &     8.97E-5     &     2.85     
 \\
\hline
$63\times 63$ &
    1.54E-4     &     3.98     &     5.50E-4     &     3.70     & 
    2.89E-4     &     4.02     &         1.11E-3     &     4.04     & 
    2.55E-6     &     3.83     &     6.57E-6     &     3.77     
 \\ \hline
\end{tabular}}
\resizebox{\textwidth}{!}{
\begin{tabular}{|c|c|c|c|c|c|c|c|c|}
\hline
\multirow{2}{*}{Finite Difference Grid} & 
\multicolumn{4}{c|}{compact finite difference \eqref{poisson-scheme2}}  &
\multicolumn{4}{c|}{Bramble-Hubbard scheme} \\ 
\cline{2-9}
 & 
 $l^2$ error & order & $l^\infty$ error & order  &
$l^2$ error & order & $l^\infty$ error & order  \\
\hline
$7\times 7$  & 
        9.88E-2     & - &         2.26E-1     & - & 
        3.14E-1     & - &         8.23E-1     & - 
 \\
\hline
$15\times 15$ & 
        5.40E-3     &     4.19     &         1.33E-2     &     4.08     & 
        1.76E-2     &     4.15     &         6.16E-2     &     3.73     
 \\
\hline
$31\times 31$ &
    3.22E-4     &     4.06    &     7.91E-4     &     4.07     & 
        3.38E-3     &     2.37     &         1.15E-2     &     2.41     
 \\
\hline
$63\times 63$ &
    1.98E-5     &     4.01     &     5.11E-5     &     3.95     & 
    3.04E-4     &     3.47     &         1.20E-3     &     3.32     
 \\ \hline
\end{tabular}}
\label{table-3}
\end{table}

\begin{table}[htbp]
\centering
\caption{Accuracy test of $Q^3$ spectral element method on uniform meshes.}
\renewcommand{\arraystretch}{2}
\scalebox{0.9}{
\begin{tabular}{|c|c|c|c|c|c|}
\hline
\multirow{1}{*}{$Q^3$ Finite Element Mesh} & \multirow{1}{*}{ Finite Difference Grid} &
 $l^2$ error & order & $l^\infty$ error & order   \\
\hline
$2\times 2$  & $5\times 5$ &
   1.18E0     &     -        &     2.61E0     &     - 
 \\
\hline
$4\times 4$ & $11\times 11$ &
    6.08E-2     &     4.28     &    1.45E-1     &    4.17     
 \\
\hline
$8\times 8$ &$23\times 23$&
    2.87E-3     &    4.40    &    7.10E-3     &     4.35     
 \\
\hline
$16\times 16$ &$47\times 47$&
    9.82E-5     &     4.87     &     2.41E-4    &     4.88     
 \\ \hline
 $32\times 32$ &$95\times 95$&
    3.12E-6     &     4.97     &     7.60E-6    &     4.99     
 \\ \hline
\end{tabular}}
\label{table-q3}
\end{table}

 \section{Concluding remarks}
 \label{sec:6}
 
We have proven that the $Q^3$ spectral element method on a uniform mesh is monotone, by proving its finite difference scheme matrix is a product of four M-matrices for two-dimensional Laplace operator.

\vspace{-0.2cm}\section*{Acknowledgments}
\vspace{-0.2cm}
Research was supported by National Science Foundation DMS-1913120.


\bibliographystyle{plain}      
\bibliography{references.bib}

\begin{thebibliography}{10}

\bibitem{bohl1979inverse}
Erich Bohl and Jens Lorenz.
\newblock Inverse monotonicity and difference schemes of higher order. a
  summary for two-point boundary value problems.
\newblock {\em aequationes mathematicae}, 19(1):1--36, 1979.

\bibitem{bramble1963fourth}
James~H Bramble.
\newblock {Fourth-order finite difference analogues of the Dirichlet problem
  for Poisson's equation in three and four dimensions}.
\newblock {\em Mathematics of Computation}, 17(83):217--222, 1963.

\bibitem{bramble1964new}
James~H Bramble and Bert~E Hubbard.
\newblock New monotone type approximations for elliptic problems.
\newblock {\em Mathematics of Computation}, 18(87):349--367, 1964.

\bibitem{bramble1962formulation}
JH~Bramble and BE~Hubbard.
\newblock {On the formulation of finite difference analogues of the Dirichlet
  problem for Poisson's equation}.
\newblock {\em Numerische Mathematik}, 4(1):313--327, 1962.

\bibitem{bramble1964finite}
JH~Bramble and BE~Hubbard.
\newblock On a finite difference analogue of an elliptic boundary problem which
  is neither diagonally dominant nor of non-negative type.
\newblock {\em Journal of Mathematics and Physics}, 43(1-4):117--132, 1964.

\bibitem{chen2001structure}
Chuanmiao Chen.
\newblock {\em Structure theory of superconvergence of finite elements (In
  Chinese)}.
\newblock Hunan Science and Technology Press, Changsha, 2001.

\bibitem{ciarlet1970discrete}
Philippe~G Ciarlet.
\newblock Discrete maximum principle for finite-difference operators.
\newblock {\em Aequationes mathematicae}, 4(3):338--352, 1970.

\bibitem{ciarlet1978finite}
Philippe~G Ciarlet.
\newblock {\em The Finite Element Method for Elliptic Problems}.
\newblock Society for Industrial and Applied Mathematics, 2002.

\bibitem{cohen2001higher}
Gary Cohen.
\newblock {\em Higher-order numerical methods for transient wave equations}.
\newblock Springer Science \& Business Media, 2001.

\bibitem{collatz1960}
Lothar Collatz.
\newblock {\em The numerical treatment of differential equations}.
\newblock Springer-Verlag, Berlin, 1960.

\bibitem{zhang2024monotonicity}
Logan~J. Cross and Xiangxiong Zhang.
\newblock On the monotonicity of {$Q^2$} spectral element method for
  {Laplacian} on quasi-uniform rectangular meshes.
\newblock {\em Communications in Computational Physics}, 35(1):160--180, 2024.

\bibitem{fornberg2015primer}
Bengt Fornberg and Natasha Flyer.
\newblock {\em A primer on radial basis functions with applications to the
  geosciences}.
\newblock SIAM, 2015.

\bibitem{fox1947some}
L~Fox.
\newblock Some improvements in the use of relaxation methods for the solution
  of ordinary and partial differential equations.
\newblock {\em Proceedings of the Royal Society of London. Series A.
  Mathematical and Physical Sciences}, 190(1020):31--59, 1947.

\bibitem{hohn1981some}
Werner H{\"o}hn and Hans~Detlef Mittelmann.
\newblock Some remarks on the discrete maximum-principle for finite elements of
  higher order.
\newblock {\em Computing}, 27(2):145--154, 1981.

\bibitem{hu2021positivity}
Jingwei Hu and Xiangxiong Zhang.
\newblock Positivity-preserving and energy-dissipative finite difference
  schemes for the {F}okker-{P}lanck and {K}eller-{S}egel equations.
\newblock {\em IMA Journal of Numerical Analysis}, 43(3):1450--1484, 2023.

\bibitem{huang2008superconvergence}
Yunqing Huang and Jinchao Xu.
\newblock Superconvergence of quadratic finite elements on mildly structured
  grids.
\newblock {\em Mathematics of computation}, 77(263):1253--1268, 2008.

\bibitem{krylov1958approximate}
Vladimir~Ivanovitch Krylov and Leonid~Vital'evitch Kantorovitch.
\newblock {\em Approximate methods of higher analysis}.
\newblock P. Noordhoff, 1958.

\bibitem{lele1992compact}
Sanjiva~K Lele.
\newblock Compact finite difference schemes with spectral-like resolution.
\newblock {\em Journal of computational physics}, 103(1):16--42, 1992.

\bibitem{MR4378595}
Hao Li, Daniel Appel\"{o}, and Xiangxiong Zhang.
\newblock Accuracy of spectral element method for wave, parabolic, and
  {S}chr\"{o}dinger equations.
\newblock {\em SIAM Journal on Numerical Analysis}, 60(1):339--363, 2022.

\bibitem{li2018high}
Hao Li, Shusen Xie, and Xiangxiong Zhang.
\newblock A high order accurate bound-preserving compact finite difference
  scheme for scalar convection diffusion equations.
\newblock {\em SIAM Journal on Numerical Analysis}, 56(6):3308--3345, 2018.

\bibitem{li2019monotonicity}
Hao Li and Xiangxiong Zhang.
\newblock {On the monotonicity and discrete maximum principle of the finite
  difference implementation of $C^0$-$ Q^2$ finite element method}.
\newblock {\em Numerische Mathematik}, pages 1--36, 2020.

\bibitem{li2019fourth}
Hao Li and Xiangxiong Zhang.
\newblock Superconvergence of high order finite difference schemes based on
  variational formulation for elliptic equations.
\newblock {\em Journal of Scientific Computing}, 82(2):36, 2020.

\bibitem{li2023-compact}
Hao Li and Xiangxiong Zhang.
\newblock {A High Order Accurate Bound-Preserving Compact Finite Difference
  Scheme for Two-Dimensional Incompressible Flow}, 2023.

\bibitem{liu2022structure}
Chen Liu, Yuan Gao, and Xiangxiong Zhang.
\newblock {Structure preserving schemes for Fokker-Planck equations of
  irreversible processes}.
\newblock {\em arXiv preprint arXiv:2210.16628}, 2022.

\bibitem{lorenz1977inversmonotonie}
Jens Lorenz.
\newblock Zur inversmonotonie diskreter probleme.
\newblock {\em Numerische Mathematik}, 27(2):227--238, 1977.

\bibitem{maday1990optimal}
Yvon Maday and Einar~M R{\o}nquist.
\newblock Optimal error analysis of spectral methods with emphasis on
  non-constant coefficients and deformed geometries.
\newblock {\em Computer Methods in Applied Mechanics and Engineering},
  80(1-3):91--115, 1990.

\bibitem{plemmons1977m}
Robert~J Plemmons.
\newblock {M-matrix characterizations. I----nonsingular M-matrices}.
\newblock {\em Linear Algebra and its Applications}, 18(2):175--188, 1977.

\bibitem{shen2021discrete}
Jie Shen and Xiangxiong Zhang.
\newblock Discrete maximum principle of a high order finite difference scheme
  for a generalized {A}llen-{C}ahn equation.
\newblock {\em Commun. Math. Sci.}, 20(5):1409--1436, 2022.

\bibitem{vejchodsky2007discrete}
Tom{\'a}{\v{s}} Vejchodsk{\`y} and Pavel {\v{S}}ol{\'\i}n.
\newblock {Discrete maximum principle for higher-order finite elements in 1D}.
\newblock {\em Mathematics of Computation}, 76(260):1833--1846, 2007.

\bibitem{wahlbin2006superconvergence}
Lars Wahlbin.
\newblock {\em Superconvergence in Galerkin finite element methods}.
\newblock Springer, 2006.

\bibitem{whiteman1975lagrangian}
JR~Whiteman.
\newblock Lagrangian finite element and finite difference methods for poisson
  problems.
\newblock In {\em Numerische Behandlung von Differentialgleichungen}, pages
  331--355. Springer, 1975.

\bibitem{xu1999monotone}
Jinchao Xu and Ludmil Zikatanov.
\newblock A monotone finite element scheme for convection-diffusion equations.
\newblock {\em Mathematics of Computation of the American Mathematical
  Society}, 68(228):1429--1446, 1999.

\end{thebibliography}

\end{document}